\documentclass[reqno]{QT}
\usepackage{amsmath,amscd,amssymb}%,epsfig}
\usepackage[all]{xy}
\usepackage[ps2pdf]{hyperref}
\usepackage{bbold}
\usepackage{graphicx}
\usepackage{psfrag}

\address[nathan.geer@usu.edu]{Nathan Geer,Mathematics \& Statistics, Utah State
University, 84322, Logan, Utah, USA}

\address[bertrand.patureau@univ-ubs.fr]{Bertrand Patureau-Mirand, LMBA
  UMR 6205, Universit\'e Europ\'eenne de Bretagne - Universit\'e de
  Bretagne-Sud, BP 573 F-56017, Vannes, France }

\address[virelizi@math.univ-montp2.fr]{Alexis Virelizier,
  D\'epartement de math\'ematiques, Universit\'e Montpellier 2, 34095
  Cedex 5, Montpellier, France} 

\newtheorem{definition}{Definition} 
\newtheorem{theorem}[definition]{Theorem}

\newtheorem{proposition}[definition]{Proposition}
\newtheorem{lemma}[definition]{Lemma}

\newtheorem{corollary}[definition]{Corollary}

\theoremstyle{definition}
\newtheorem{rem}[definition]{Remark}

\newcommand{\kt}{$\Bbbk$\nobreakdash-\hspace{0pt}}
\newcommand{\kk}{\Bbbk}
\newcommand{\Ob}{\mathrm{Ob}}
\newcommand{\oo}{\mathcal{O}}
\newcommand{\R}{\mathbb{R}}
\newcommand{\opp}{\mathrm{op}}
\newcommand{\rev}{\mathrm{rev}}
\newcommand{\cat}{\mathcal{C}}
\newcommand{\Graph}{\mathcal{G}}
\newcommand{\ideal}{\mathcal{I}}

\newcommand{\ev}{\mathrm{ev}}
\newcommand{\coev}{\mathrm{coev}}
\newcommand{\tev}{\widetilde{\mathrm{ev}}}
\newcommand{\tcoev}{\widetilde{\mathrm{coev}}}
\newcommand{\unit}{\mathbb{1}}
\newcommand{\qd}{\ensuremath{\mathsf{d}}}
\newcommand{\brk}[1]{\left\langle{#1}\right\rangle}
\newcommand{\A}{{\mathsf{A}}}
\newcommand{\B}{{\mathsf{B}}}
\newcommand{\tzz}{{\mathsf{t}}}

\newcommand{\tet}{\mathcal{T}}
\newcommand{\ang}[1]{{\left\langle{#1}\right\rangle}}
\newcommand{\co}{\colon}
\newcommand{\ep}{\varepsilon}
\newcommand{\Proj}{\operatorname{\mathsf{Proj}}}
\newcommand{\cl}{\operatorname{\mathrm{cl}}}
\DeclareMathOperator{\End}{{End}}
\DeclareMathOperator{\Hom}{Hom}
\DeclareMathOperator{\Id}{id}
\DeclareMathOperator{\id}{id}
\DeclareMathOperator{\tr}{tr}

\newcommand{\rsdraw}[3]{
  \raisebox{-#1\height}{\scalebox{#2}{\includegraphics{#3.eps}}}}

\title{Traces on ideals in pivotal categories}

\author[N Geer, B Patureau-Mirand, A Virelizier]{Nathan Geer, Bertrand
  Patureau-Mirand, Alexis Virelizier\thanks{The work of Nathan Geer
    has been partially supported by the NSF grants DMS-0706725 and
    DMS-0968279.  Alexis Virelizier was partially supported by the ANR
    grant GESAQ}}

\begin{document}

\maketitle

\begin{abstract}
  We construct invariants of $\cat$-colored spherical graphs from
  traces on ideals in a pivotal category $\cat$. Then we provide a
  systematic approach to defining such traces from ambidextrous and
  spherical traces on a class of objects of $\cat$. This extends the
  notion of an ambidextrous object of a braided category given by the
  first two authors in a previous work.
\end{abstract}

% \setcounter{tocdepth}{1}
% \tableofcontents

\section*{Introduction}
%%%%%%%%%%%%%%%%%%%%%%%%%%%%%%%%%%%%%%%%%%%%%%%%%%%%%%%%%%%%%%%%%%%%%%%%%%%%%%
%%%%%%%%%%%%%%%%%%%%%%%%%%%%%%%%%%%%%%%%%%%%%%%%%%%%%%%%%%%%%%%%%%%%%%%%%%%%%%
Many constructions of quantum topology rely on tensor categories with duality
and more precisely on categorical traces of such categories.  When these
categories are not semi-simple, the categorical traces are often degenerate and these
constructions become trivial.  In \cite{GKP1}, \cite{GPT2} it has been shown
that in non-semi-simple ribbon categories, non-trivial traces can exist.
The study of these traces leads to new
interesting quantum invariants of links and 3-manifolds.  The goal of this
paper is to generalize this study within the context of pivotal categories.
In particular, we develop a theory of modified traces on ideals in pivotal categories
and show that such traces lead to topological invariants of planar graphs.

Let $\cat$ be a pivotal \kt category.   We call a class $\ideal$ of objects of $\cat$ a \emph{left ideal} (resp.\@ a \emph{right ideal}) if it is closed
under retraction and  left (resp.\@ right) tensor multiplication by objects of $\cat$.  An \emph{ideal} is a two-side ideal.    A \emph{left trace} (resp.\@ a \emph{right trace}) on a left (resp.\@ right)  ideal $\ideal$
is a family of \kt linear functions
$$\tzz = \{\tzz_{V}\colon \End_{\cat}(V) \to \kk \}_{V \in \ideal}$$
 which is suitably compatible with the tensor product and composition of
morphisms (cf.\ Section~\ref{SS-sided-traces}).  A \emph{trace} on an ideal $\ideal$ is a family $\tzz = \{\tzz_{V}\}_{V \in \ideal}$ which is both a left and right trace on $\ideal$.

Let us explain how left traces and traces lead to topological invariants of certain graphs.  Similar invariants exist for right traces.
We consider planar and spherical graphs (i.e., graphs in $\R^2$ and $S^2=\R^2\cup\{\infty\}$, respectively).   If $\oo$ is a class of objects of $\cat$, then a \emph{$\oo$-colored graph} in $\R^2$ (resp.\@ $S^2$) is a ribbon graph embedded in $\R^2$ (resp.\@ $S^2$) whose edges are colored by elements of $\oo$ and coupons are colored by morphisms in $\cat$.   Let $\Gamma$ be a planar $\cat$-colored graph.    By a \emph{left} \emph{cutting presentation} of  $\Gamma$, we mean a $\cat$-colored 1-1-ribbon graph $T$ in $\R\times [0,1]$ such that $\Gamma$ is the left  closure of $T$.    Let $\tzz$ be a left trace on a left ideal $\ideal$ of $\cat$.   We say $\Gamma$ is  \emph{$\ideal$-admissible} if it admits a left cutting presentation $T$ whose open component is colored by an object $V$ of $\ideal$.   For such a  left cutting presentation we set $$F_\tzz(\Gamma)=\tzz_{V}(f_T)$$
where $f_T\colon V\to V$ is the morphism defined by~$T$ via Penrose calculus
(see Subsection \ref{sec-penrose}).
In Section \ref{S:InvClosedGraph} we prove the following statements:
 \begin{enumerate}
\renewcommand{\labelenumi}{{\rm (\roman{enumi})}}
\item Let $\Gamma$ be an $\ideal$-admissible planar graph. 
Then the scalar $F_\tzz(\Gamma)$ is well defined (that is, independent of the choice of the left cutting presentation $T$ of $\Gamma$).  The assignment $\Gamma \mapsto  F_\tzz(\Gamma)$ is then an isotopy invariant of $\ideal$-admissible planar graphs $\Gamma$.
\item \label{I:InvSphwithA} 
The left trace $\tzz$ on the left ideal $\ideal$ determines a class $\A\subset \ideal$ (see Equation \eqref{E:DefA}) such that for any $\A$-colored spherical graph~$\Gamma$
the scalar  $F_\tzz(\Gamma)$ is well defined.  The assignment $\Gamma \mapsto  F_\tzz(\Gamma)$ is then an isotopy invariant of $\A$-colored spherical graphs $\Gamma$.
\item If $\ideal$ is an ideal and $\tzz$ is a trace on $\ideal$, then $\A=\ideal$ and so the assignment $\Gamma \mapsto F_\tzz(\Gamma)$ is a well defined isotopy invariant of $\A$-colored spherical graphs $\Gamma$.  Moreover, this assignment extends to an isotopy invariant of $\ideal$-admissible spherical graphs $\Gamma$.
\end{enumerate}

In Section \ref{S:AmbiTrace} we give a systematic approach to defining traces on ideals. More precisely, let $\oo$ be a class of  objects of $\cat$.  We define a \emph{left ambidextrous trace}  on $\oo$ as a family $t=\{t_X\colon \End_\cat(X)\rightarrow \kk\}_{X\in\oo}$ of \kt linear forms satisfying
\begin{equation*}
t_X\left (\, \psfrag{Y}[Bc][Bc]{\scalebox{.8}{$X'$}} \psfrag{X}[Bc][Bc]{\scalebox{.8}{$X$}} \psfrag{f}[Bc][Bc]{\scalebox{.9}{$f$}} \rsdraw{.45}{.9}{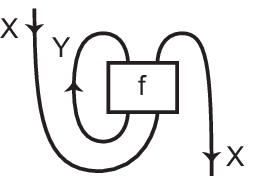}  \, \right ) = t_{X'} \left (\, \psfrag{Y}[Bc][Bc]{\scalebox{.8}{$X'$}} \psfrag{X}[Bc][Bc]{\scalebox{.8}{$X$}} \psfrag{f}[Bc][Bc]{\scalebox{.9}{$f$}} \rsdraw{.45}{.9}{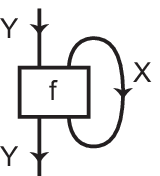}  \right )
\end{equation*}
for all $X,X' \in \oo$ and $f \in \End_\cat(X' \otimes X^*)$.  Then we show that left ambidextrous traces on $\oo$ bijectively correspond to left traces on the left ideal generated by $\oo$.  We also give a notion of a spherical trace on $\oo$ and show that these traces bijectively correspond to traces on the ideal generated by $\oo$.

The invariant of spherical graphs described above in (ii) relies on a certain set~$\A$ of objects defined from a
one-sided trace on a one-sided ideal.  In Section \ref{S:Slope} we give a
characterization of $\A$, in terms of the slope, when the one-sided ideal is
the ideal of projective objects.  In \cite{GP4}, 
the slope and left ambidextrous traces are used to construct 3-manifold invariants from the categories of
finite dimensional modules over the non-restricted quantum groups associated
to simple Lie algebras introduced and studied by De~Concini, Kac, Procesi,
Reshetikhin, and Rosso in \cite{DK}, \cite{DPRR}.
The categories of finite dimensional representations of these quantum groups
cannot be braided as their tensor product is not commutative up to
isomorphism.

In \cite{GKP2} the results of this paper are used
to prove
that traces on the ideal of projective modules exist for factorizable ribbon Hopf algebras, modular representations of
  finite groups and their quantum doubles, complex and modular Lie
  (super)algebras, the $(1,p)$ minimal model in conformal field theory, and
  quantum groups at a root of unity.  In all these examples the usual trace restricted to the ideal of projective modules is zero but the modified trace suggested in this paper is non-zero.

Throughout the paper, we fix  a commutative ring $\kk$.

\section{Pivotal categories}\label{S:PivCat}
In this section we recall some well-known properties concerning pivotal categories.   Given a category $\cat$, the notation $X\in \cat$ means that $X$ is an object of~$\cat$.

\subsection{Pivotal categories}\label{sect-picotal-cat}
Recall that a \emph{pivotal} (or \emph{sovereign}) category is a (strict) monoidal category $\cat$, with unit object~$\unit$, such that to each
object $X\in \cat$ there are associated a \emph{dual object}~$X^*\in \cat$ and four morphisms
\begin{align*}
& \ev_X \colon X^*\otimes X \to\unit,  \quad \coev_X\colon \unit  \to X \otimes X^*,\\
&   \tev_X \colon X\otimes X^* \to\unit, \quad   \tcoev_X\colon \unit  \to X^* \otimes X,
\end{align*}
such that $(\ev_X,\coev_X)$ is a left
duality for $X$, $( \tev_X,\tcoev_X)$ is a right
duality for~$X$, and the induced left and right dual functors coincide as monoidal functors (see~\cite{Malt}). In particular, the left dual and right dual of a morphism $f\colon X \to Y$ in $\cat$ coincide:
\begin{align*}
f^*&= (\ev_Y \otimes  \Id_{X^*})(\Id_{Y^*}  \otimes f \otimes \Id_{X^*})(\Id_{Y^*}\otimes \coev_X)\\
 &= (\Id_{X^*} \otimes \tev_Y)(\Id_{X^*} \otimes f \otimes \Id_{Y^*})(\tcoev_X \otimes \Id_{Y^*}) \colon Y^*\to X^*,
\end{align*}
and $$
\phi=\{\phi_X=(\tev_{X}\otimes\Id_{X^{**}})(\Id_X\otimes\coev_{X^{*}})\colon X\to
X^{**}\}_{X \in \cat}
$$
is a monoidal natural isomorphism, called the \emph{pivotal structure}.

\subsection{Traces and dimensions}\label{sec-traces} Let $\cat$ be pivotal category. Recall that $\End_\cat(\unit)$ is a commutative monoid. The
\emph{left trace} $\tr_l(f)\in\End_\cat(\unit)$ and the \emph{right trace} $\tr_r(f) \in
\End_\cat(\unit)$ of an endomorphism $f $ of an object $X$ of $\cat$ are defined by
$$\tr_l(f)=\ev_X(\Id_{X^*} \otimes f) \tcoev_X  \quad {\text
{and}}\quad
 \tr_r(f)=\tev_X( f \otimes \Id_{X^*}) \coev_X  .
$$
Both traces are symmetric: $\tr_l(gh)=\tr_l(hg)$ and $
\tr_r(gh)=\tr_r(hg)$ for any morphisms $g\colon X \to Y$ and   $h\colon Y
\to X$  in $\cat$. Also $\tr_l(f)=\tr_r(f^*)=\tr_l(f^{**})$ for
any endomorphism $f $ of an object (and similarly  with $l,r$
exchanged). If
\begin{equation}\label{special}
\alpha\otimes \Id_X= \Id_X\otimes \alpha \quad \text{for all $\alpha\in
\End_{\cat}(\unit)$ and $X\in \cat$,}
\end{equation}
then the traces $\tr_l,\tr_r$
are $\otimes$-multiplicative: $$ \tr_l(f\otimes g)=\tr_l(f) \,
\tr_l(g) \quad \text{and} \quad \tr_r(f\otimes g)=\tr_r(f) \, \tr_r(g) $$ for all
endomorphisms $f,g$ of objects of~$\cat$.

The \emph{left} and \emph{right dimensions} of   $X\in \Ob (\cat)$
are defined by $ \dim_l(X)=\tr_l(\Id_X) $ and $
\dim_r(X)=\tr_r(\Id_X) $. Clearly,
$\dim_l(X)=\dim_r(X^*)=\dim_l(X^{**})$ (and similarly  with
$l,r$ exchanged). Note that isomorphic objects have the same
dimensions and $\dim_l(\unit)=\dim_r(\unit)=\Id_{\unit}$. If the category~$\cat$ satisfies \eqref{special}, then  left and right dimensions are $\otimes$-multiplicative: $\dim_l (X\otimes Y)= \dim_l (X)
\dim_l (Y)$ and $\dim_r (X\otimes Y)= \dim_r (X) \dim_r (Y)$ for any
$X,Y\in \cat$.

\subsection{Penrose graphical calculus}\label{sec-penrose} We represent morphisms in a category $\cat$ by plane   diagrams to be read from the bottom to the top.
The  diagrams are made of   oriented arcs colored by objects of
$\cat$  and of boxes (called \emph{coupons}) colored by morphisms of~$\cat$.  The arcs connect
the boxes and   have no mutual intersections or self-intersections.
The identity $\Id_X$ of $X\in \cat$, a morphism $f\colon X \to Y$,
and the composition of two morphisms $f\colon X \to Y$ and $g\colon Y \to
Z$ are represented as follows:
\begin{center}
\psfrag{X}[Bc][Bc]{\scalebox{.7}{$X$}} \psfrag{Y}[Bc][Bc]{\scalebox{.7}{$Y$}} \psfrag{h}[Bc][Bc]{\scalebox{.8}{$f$}} \psfrag{g}[Bc][Bc]{\scalebox{.8}{$g$}}
\psfrag{Z}[Bc][Bc]{\scalebox{.7}{$Z$}} $\Id_X=$ \rsdraw{.45}{.9}{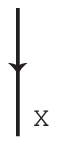}\,,\quad $f=$ \rsdraw{.45}{.9}{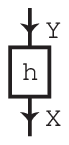} ,\quad \text{and} \quad $gf=$ \rsdraw{.45}{.9}{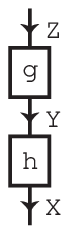}\,.
\end{center}
  If $\cat$ is
monoidal, then the monoidal product of two morphisms $f\colon X \to Y$
and $g \colon U \to V$ is represented by juxtaposition:
\begin{center}
\psfrag{X}[Bc][Bc]{\scalebox{.7}{$X$}} \psfrag{h}[Bc][Bc]{\scalebox{.8}{$f$}}
\psfrag{Y}[Bc][Bc]{\scalebox{.7}{$Y$}}  $f\otimes g=$ \rsdraw{.45}{.9}{morphism} \psfrag{X}[Bc][Bc]{\scalebox{.8}{$U$}} \psfrag{g}[Bc][Bc]{\scalebox{.8}{$g$}}
\psfrag{Y}[Bc][Bc]{\scalebox{.7}{$V$}} \rsdraw{.45}{.9}{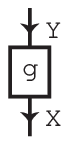}\,.
\end{center}
In a pivotal category, if an arc colored by $X$ is oriented upwards,
then the corresponding object   in the source/target of  morphisms
is $X^*$. For example, $\Id_{X^*}$  and a morphism $f\colon X^* \otimes
Y \to U \otimes V^* \otimes W$  may be depicted as:
\begin{center}
 $\Id_{X^*}=$ \, \psfrag{X}[Bl][Bl]{\scalebox{.7}{$X$}}
\rsdraw{.45}{.9}{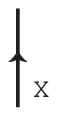} $=$  \,
\psfrag{X}[Bl][Bl]{\scalebox{.7}{$X^*$}}
\rsdraw{.45}{.9}{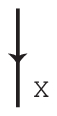}  \quad and \quad
\psfrag{X}[Bc][Bc]{\scalebox{.7}{$X$}}
\psfrag{h}[Bc][Bc]{\scalebox{.8}{$f$}}
\psfrag{Y}[Bc][Bc]{\scalebox{.7}{$Y$}}
\psfrag{U}[Bc][Bc]{\scalebox{.7}{$U$}}
\psfrag{V}[Bc][Bc]{\scalebox{.7}{$V$}}
\psfrag{W}[Bc][Bc]{\scalebox{.7}{$W$}} $f=$
\rsdraw{.45}{.9}{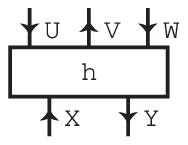} \,.
\end{center}
The duality morphisms   are depicted as follows:
\begin{center}
\psfrag{X}[Bc][Bc]{\scalebox{.7}{$X$}} $\ev_X=$ \rsdraw{.45}{.9}{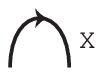}\,,\quad
 $\coev_X=$ \rsdraw{.45}{.9}{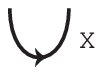}\,,\quad
$\tev_X=$ \rsdraw{.45}{.9}{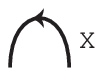}\,,\quad
\psfrag{C}[Bc][Bc]{\scalebox{.7}{$X$}} $\tcoev_X=$
\rsdraw{.45}{.9}{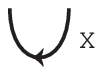}\,.
\end{center}
The dual of a morphism $f\colon X \to Y$ and the
  traces of a morphism $g\colon X \to X$ can be depicted as
follows:
\begin{center}
\psfrag{X}[Bc][Bc]{\scalebox{.7}{$X$}} \psfrag{h}[Bc][Bc]{\scalebox{.8}{$f$}}
\psfrag{Y}[Bc][Bc]{\scalebox{.7}{$Y$}} \psfrag{g}[Bc][Bc]{\scalebox{.8}{$g$}}
$f^*=$ \rsdraw{.45}{.9}{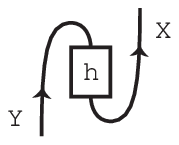}$=$ \rsdraw{.45}{.9}{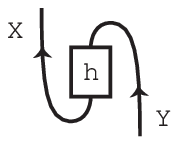}\quad \text{and} \quad
$\tr_l(g)=$ \rsdraw{.45}{.9}{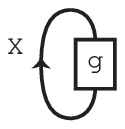}\,,\quad  $\tr_r(g)=$ \rsdraw{.45}{.9}{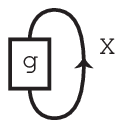}\,.
\end{center}
  If $\cat$ is pivotal, then    the morphisms represented by the diagrams
are invariant under isotopies of the diagrams in the plane keeping
fixed the bottom and   top endpoints.

\subsection{Partial traces}
Let $\cat$ be a pivotal category. For $X,Y,Z\in\cat$, the \emph{left partial trace} (with respect to $X$) is the map $\tr_l^{X} \colon \Hom_\cat(X \otimes Y, X \otimes Z) \to \Hom_\cat(Y,Z) $ defined for $f \in \Hom_\cat(X \otimes Y, X \otimes Z)$
by
$$
\tr_l^X(f)=(\ev_X \otimes \Id_Z)(\Id_{X^*} \otimes f)(\tcoev_X \otimes \Id_Y) = \psfrag{Y}[Bc][Bc]{\scalebox{.8}{$Y$}} \psfrag{Z}[Bc][Bc]{\scalebox{.8}{$Z$}} \psfrag{X}[Bc][Bc]{\scalebox{.8}{$X$}} \psfrag{f}[Bc][Bc]{\scalebox{.9}{$f$}} \rsdraw{.45}{.9}{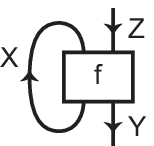} \,.
$$
Likewise, the \emph{right partial trace} (with respect to $X$) is the map
$\tr_r^X \colon \Hom_\cat(Y \otimes X, Z \otimes X) \to \Hom_\cat(Y,Z)$
defined, for $g \in  \Hom_\cat(Y \otimes X, Z \otimes X)$ by
$$
\tr_r^X(g)=(\Id_Z \otimes \tev_X)(g \otimes \Id_{X^*})(\Id_Y \otimes 
\coev_X) 
=\psfrag{Y}[Bc][Bc]{\scalebox{.8}{$Y$}} \psfrag{Z}[Bc][Bc]{\scalebox{.8}{$Z$}} \psfrag{X}[Bc][Bc]{\scalebox{.8}{$X$}} \psfrag{f}[Bc][Bc]{\scalebox{.9}{$g$}} \rsdraw{.45}{.9}{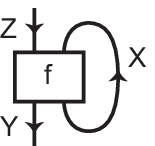} \,.
$$

Note that $\tr_r^{X^*}(f^*)=\bigl(\tr_l^X(f) \bigr)^*$ and $\tr_l^{X^*}(g^*)=\bigl(\tr_r^X(g) \bigr)^*$.
In particular
\begin{align*}
 &\tr_l^{X^{**}}(f^{**})=\bigl(\tr_l^X(f) \bigr)^{**}=\phi_Y \tr_l^X(f) \phi_Y^{-1},\\
 &\tr_r^{X^{**}}(g^{**})=\bigl(\tr_r^X(g) \bigr)^{**}=\phi_X \tr_r^X(g) \phi_X^{-1},
\end{align*}
where $\phi=\{\phi_X\colon X\to X^{**}\}_{X \in \cat}$ is the pivotal structure of $\cat$.

Note that if $f\colon X \to X$ is an endomorphism in $\cat$, then
$$\tr_l^X(f)=\tr_l(f), \quad \tr_r^X(f)=\tr_r(f), \quad \text{and} \quad \tr_l^\unit(f)=f=\tr_r^\unit(f).$$

\subsection{Spherical categories}\label{sec-spherical}  A \emph{spherical category} is a pivotal category whose left and
right traces are equal, i.e.,  $\tr_l(f)=\tr_r(f)$ for every
endomorphism $f$ of an object. Then $\tr_l(f)$ and $ \tr_r(f)$ are
denoted $\tr(f)$ and called the \emph{trace of $f$}. Similarly, the
left and right dimensions of an object~$X$ are equal, denoted   $\dim(X)$,
and called the \emph{dimension of $X$}.

For spherical categories, the corresponding Penrose graphical
calculus has the following property:   the morphisms represented by
 diagrams are invariant under isotopies of   diagrams in the 2-sphere $S^2=\R^2\cup
\{\infty\}$, i.e., are preserved under isotopies pushing   arcs of
the diagrams across~$\infty$.  For example, the diagrams above
representing $\tr_l(f)$ and $\tr_r(f)$ are related by such an
isotopy. The condition $\tr_l(f)=\tr_r(f)$ for all $f$ is therefore
necessary (and in fact sufficient) to ensure this property.

\subsection{Linear categories}\label{sphesphe}
A \emph{monoidal \kt category} is a monoidal category $\cat$ such that its
 hom-sets are (left) \kt modules, the composition and monoidal product of morphisms are \kt bilinear,
  and  $\End_\cat(\unit)$ is a free \kt module of rank one.
Then the map $\kk \to \End_\cat(\unit), k \mapsto k \, \Id_\unit$  is a
\kt algebra isomorphism. It is used to identify $\End_\cat(\unit)=\kk$.

A pivotal \kt category satisfies \eqref{special}. Therefore, the traces $\tr_l,\tr_r$ and the
dimensions $\dim_l, \dim_r$ in such a category are
$\otimes$-multiplicative. Clearly,   $\tr_l,\tr_r$ are  \kt linear.

An object $X$ of a monoidal \kt category $\cat$ is
\emph{simple} if $\End_\cat(X)$ is a free \kt module of rank 1.
Equivalently, $X$ is simple if the \kt homomorphism $\kk \to
\End_\cat(X),\, k   \mapsto  k\, \Id_X$  is an isomorphism. By the
definition of a monoidal \kt category, the unit object $\unit$ is
simple. If $X$ is a simple object of $\cat$, we
denote by $\brk{\,}_V \colon \End_\cat(V) \to \kk$ the inverse of the \kt
linear isomorphism $\kk \to \End_\cat(V)$ defined by $k \mapsto k\, \Id_V$.

\section{Traces on ideals}\label{sect-ideaux}
In this section we introduce the notion of one-side ideals and one-side traces on one-side ideals.
We also study some basic properties of such ideals and traces. 

The notation $\oo \subset \cat$ means that $\oo$ is a class of objects of $\cat$.
If $\cat$ is a monoidal category, we denote by $\cat^\rev$ the category $\cat$ with opposite monoidal product defined by $X \otimes^\opp Y=Y \otimes X$ for $X,Y \in \cat$.

\subsection{Ideals} By a \emph{retract} of an object $X$ of a category $\cat$, we mean an object $U$ of~$\cat$ such that there exist morphisms $p\colon X \to U$ and $q\colon U \to X$ verifying $pq=\Id_U$.

A class $\oo \subset \cat$ is said to be \emph{closed under retraction} if any retract (in $\cat$) of an object of $\oo$ belongs to $\oo$.

A class $\oo$ of objects of a monoidal category $\cat$ is said to be \emph{closed under left} (resp.\@ \emph{right}) \emph{multiplication} if $Y\otimes X\in \oo$ (resp.\@ $X\otimes Y\in \oo$) for all $X\in \oo$ and $Y\in\cat$.

By a \emph{left} (resp.\@ \emph{right}) \emph{ideal} of a monoidal category $\cat$, we mean a class $\ideal \subset \cat$ which is closed under retraction and under left (resp.\@ right) multiplication. By an \emph{ideal} of a monoidal category $\cat$, we mean a class $\ideal \subset \cat$ which is both a left and right ideal.

Note that the closure under retraction implies that a left (resp.\@ right) ideal $\ideal$ of a monoidal category $\cat$ is \emph{replete}, meaning that if $Y \in \cat$ is isomorphic (in $\cat$) to some $X \in \ideal$, then $Y \in \ideal$. In particular, if $\cat$ is braided, then any left (resp.\@ right) ideal is a (two-sided) ideal.

\begin{lemma}\label{lem-dual-ideal}
Let $\cat$ be a pivotal category.
  \begin{enumerate}
\renewcommand{\labelenumi}{{\rm (\alph{enumi})}}
\item A left (resp.\@ right) ideal of $\cat$ is closed under biduality.
\item  An ideal of $\cat$ is closed under duality.
\end{enumerate}
\end{lemma}
\begin{proof}
Part (a) follows from the facts that a left (resp.\@ right) ideal is replete and that the bidual $X^{**}$ of an object $X\in\cat$  is isomorphic to $X$ (via the pivotal structure).

Let us prove Part (b).
Let $\ideal$ be an ideal of $\cat$. Given $X\in\ideal$, set $p=\ev_X \otimes \id_{X^*}$ and $q=\id_{X^*} \otimes \coev_X$. Then $pq=\id_{X^*}$. Thus $X^*$ is a retract of $X^* \otimes X \otimes  X^*\in \ideal$ and so belongs to $\ideal$.
\end{proof}

For a class of objects $\ideal$ in a pivotal category $\cat$, we set
$$
\ideal^{*}=\{Y \in \cat\, | \,  \exists \, X \in \ideal, \; Y \simeq X^*  \} .
$$
Note that if $\ideal$ is a left or right ideal of $\cat$, then
$$\ideal^{*}= \{Y \in \cat\, | \, Y^* \in \ideal \}$$
(since ideals are replete and $Y\simeq Y^{**}$ for any $Y \in \cat$).
\begin{lemma}\label{L:id*}
  Let $\ideal$ be a replete class of objects of a pivotal category $\cat$.
 \begin{enumerate}
\renewcommand{\labelenumi}{{\rm (\alph{enumi})}}
\item $\ideal$ is a left (resp.\@ right) ideal if and only if
  $\ideal^{*}$ is a right (resp.\@ left) ideal.
\item $\ideal$ is an ideal if and only if $\ideal$ is a left (or right) ideal and $\ideal^{*}=\ideal$.
\end{enumerate}
\end{lemma}
\begin{proof}
Part (a) follows from the fact that $(X\otimes Y)^{*}\simeq
  Y^{*}\otimes X^{*}$ and from the equivalence asserting that an object $U$ is a retract of an object $X$ if
  and only if $U^{*}$ is a retract of $X^{*}$. Let us prove Part (b). If $\ideal$ is a left (resp.\@ right) ideal
  satisfying $\ideal=\ideal^{*}$ then, by Part (a), it is a right (resp.\@ left) ideal, and so an ideal. Conversely if $\ideal$ is an ideal of $\cat$, then $\ideal^{*}=\ideal$ by Lemma \ref{lem-dual-ideal}(b).
  \end{proof}

\subsection{Traces on ideals}\label{SS-sided-traces}
Let $\cat$ be a pivotal \kt category. A \emph{left trace} on a left ideal  $\ideal$ of $\cat$ is
a family $\tzz=\{\tzz_X\colon \End_\cat(X)\rightarrow \kk\}_{X\in\ideal}$
of \kt linear forms
such that
\begin{equation}\label{defltrace}
\tzz_{Y\otimes X}(f)=\tzz_X(\tr_l^Y(f)) \quad \text{and} \quad \tzz_V(gh)=\tzz_U(hg)
\end{equation}
for any $f\in \End_\cat(Y\otimes X)$,  $g \in \Hom_\cat(U,V)$, and $h \in \Hom_\cat(V,U)$, with $X,U,V\in \ideal$ and $Y\in \cat$.

A \emph{right trace} on a right ideal  $\ideal$ of $\cat$ is
a family $\tzz=\{\tzz_X\colon \End_\cat(X)\rightarrow \kk\}_{X\in\ideal}$
of \kt linear forms
such that
\begin{equation}\label{defrtrace}
\tzz_{X\otimes Z}(f)=\tzz_X(\tr_r^Z(f)) \quad \text{and} \quad \tzz_V(gh)=\tzz_U(hg)
\end{equation}
for any $f\in \End_\cat(X\otimes Z)$,  $g \in \Hom_\cat(U,V)$, and $h \in \Hom_\cat(V,U)$, with $X,U,V\in \ideal$ and $Z\in \cat$.

A \emph{trace} on a ideal  $\ideal$ of $\cat$ is
a family $\tzz=\{\tzz_X\colon \End_\cat(X)\rightarrow \kk\}_{X\in\ideal}$
of \kt linear forms
which is both a left and right trace on $\ideal$. Note that a trace $\tzz$ on an ideal $\ideal$ satisfies
$$
\tzz_{Y \otimes X\otimes Z}(f)=\tzz_X(\tr_l^Y\tr_r^Z(f))
$$
for any $f\in \End_\cat(Y \otimes X\otimes Z)$ with $X\in \ideal$ and $Y,Z\in \cat$.

\begin{lemma}\label{lem-dual-trace}
  \begin{enumerate}
\renewcommand{\labelenumi}{{\rm (\alph{enumi})}}
\item  If $\tzz$ is a left (resp.\@ right) trace on a left (resp.\@ right) ideal  $\ideal$ of~$\cat$, then
$$
\tzz_{X^{**}}(f^{**})=\tzz_X(f)
$$ for all $X \in \ideal$ and $f\in \End_\cat(X)$.
\item  If $\tzz$ is a  trace on an ideal  $\ideal$ of $\cat$, then
$$
\tzz_{X^{*}}(f^{*})=\tzz_X(f)
$$ for all $X \in \ideal$ and $f\in \End_\cat(X)$.
\item If $\tzz$ is a left (resp.\@ right) trace on a left (resp.\@ right) ideal
  $\ideal$ of $\cat$, then the family of \kt linear forms
  $\tzz^{\vee}=\{\tzz^{\vee}_{X}\colon \End_\cat(X)\rightarrow \kk\}_{X\in\ideal^{*}}$,
  defined by $$\tzz^{\vee}_X(f)=\tzz_{X^{*}}(f^{*}),$$ is a right (resp.\@ left) trace
  on $\ideal^{*}$.
\end{enumerate}
\end{lemma}
\begin{proof}
Let us prove Part (a). Denote by $\phi$ the pivotal structure of $\cat$  (see Section~\ref{sect-picotal-cat}).
Recall $X^{**} \in \ideal$ (see Lemma~\ref{lem-dual-ideal}) and $f^{**}=\phi_X f \phi_X^{-1}$. Therefore,
$$\tzz_{X^{**}}(f^{**})=\tzz_{X^{**}}(\phi_X f \phi_X^{-1})=\tzz_{X}( f \phi_X^{-1}\phi_X)=\tzz_X(f).$$

Let us prove Part (b). Recall that $X^{*} \in \ideal$ (see Lemma~\ref{lem-dual-ideal}). In particular, $X^* \otimes X \in \ideal$ and $X^* \otimes X \otimes X^* \in \ideal$.  Set $g=(\ev_X \otimes \Id_{X^*})(\Id_{X^*} \otimes f \otimes \Id_{X^*})$ and $h=\id_{X^*} \otimes \coev_X$.
We have: $gh=f^*$ and $\tr_l^{X^*}(\tr_r^{X^*}(hg))=f$. Therefore,
\begin{align*}
\tzz_{X^*}(f^*)&=\tzz_{X^*}(gh)=\tzz_{X^*\otimes X \otimes X^*}(hg)\\
&=\tzz_{X^*\otimes X}\bigl(\tr_r^{X^*}(hg)\bigr) =\tzz_X\bigl(\tr_l^{X^*}(\tr_r^{X^*}(hg))\bigr)=\tzz_X(f)
\end{align*}
by the properties of a (two-sided) trace.

Let us prove the left version of Part (c), from which the right version can be deduced by using $\cat^\rev$. For $g \in \Hom_\cat(U,V)$ and $h \in \Hom_\cat(V,U)$,
with $U,V\in \ideal^{*}$, we have:
$$
\tzz^{\vee}_{V}(gh)=\tzz_{V^{*}}((gh)^{*})=\tzz_{V^{*}}(h^{*}g^{*})=\tzz_{U^{*}}(g^{*}h^{*})
=\tzz_{U^{*}}((hg)^{*})=\tzz^{\vee}_{U}(hg).
$$  Now let $f\in\End_\cat(Y\otimes
Z)$ with $X \in \ideal^*$ and $Z\in\cat$. Let $\varphi\colon (X \otimes Z)^* \to Z^* \otimes X^*$ be the canonical isomorphism. Then $\tr_l^{Z^*}\!(\varphi f^* \varphi^{-1})=\bigl(\tr_r^{Z}(f)\bigr)^*$ and so
\begin{align*}
\tzz^{\vee}_{X\otimes Z}(f)&=\tzz_{(X \otimes Z)^*}(f^*)=\tzz_{Z^* \otimes X^*}(\varphi f^*\varphi^{-1})\\
&=\tzz_{X^*}\bigl(\tr_r^{Z^*}(\varphi f^*\varphi^{-1})\bigr)=\tzz_{X^*}\bigl((\tr_r^{Z}(f))^*\bigr)=\tzz_{X}^\vee(\tr_r^{Z}(f))
\end{align*}
since $\tzz$ is a left trace and $X^* \in \ideal$. Hence, $\tzz^{\vee}$ is a right trace.
\end{proof}

\begin{lemma}\label{lem-trace-on-ideal-are-ambi}
Let $\cat$ be a pivotal \kt category. Denote by $\phi$ the pivotal structure of $\cat$ (see Section~\ref{sect-picotal-cat}). Then:
  \begin{enumerate}
\renewcommand{\labelenumi}{{\rm (\alph{enumi})}}
 \item  If  $\tzz$ is a left trace on a left ideal  $\ideal$ of~$\cat$ then  
\begin{equation}\label{lambi}
\tzz_X\bigl(\phi_X^{-1} \bigl(\tr_l^{X'}\!(f)\bigr)^* \phi_X\bigr)=\tzz_{X'}\bigl(\tr_r^{X^*}\!(f)\bigr)
\end{equation}
for all $X,X' \in \ideal$ and $f \in \End_\cat(X' \otimes X^*)$.
 \item  If $\tzz$ is a right trace on a right ideal  $\ideal$ of~$\cat$ then 
\begin{equation}\label{rambi}
\tzz_X\bigl(\phi_X^{-1} \bigl(\tr_r^{X'}\!(g)\bigr)^* \phi_X\bigr)=\tzz_{X'}\bigl(\tr_l^{X^*}\!(g)\bigr)
\end{equation}
 for all $X,X' \in \ideal$ and $g \in \End_\cat(X^* \otimes X')$.
 \item If $\tzz$ is a trace on an ideal $\ideal$ of~$\cat$ then 
 \begin{equation}\label{ambi1}
\tzz_X\bigl(\phi_X^{-1} \bigl(\tr_l^{X'\otimes Y}\!(f)\bigr)^* \phi_X\bigr)=\tzz_{X'}\bigl(\tr_r^{Y \otimes X^*}\!(f)\bigr),
\end{equation}
and 
\begin{equation}\label{ambi2}
\tzz_X\bigl(\phi_X^{-1} \bigl(\tr_r^{Y\otimes X'}\!(g)\bigr)^* \phi_X\bigr)=\tzz_{X'}\bigl(\tr_l^{X^* \otimes Y}\!(g)\bigr),
\end{equation}
for all $X,X' \in \ideal$, $Y \in \cat$, $f \in \End_\cat(X' \otimes Y \otimes X^*)$, and $g\in \End_\cat(X^*\otimes Y \otimes X')$.
\end{enumerate}
\end{lemma}
\begin{proof}
Let us prove Part (a). Let $f \in \End_\cat(X' \otimes X^*)$ where $X,X' \in \ideal$. Set
$$
\alpha=(\Id_{X'^*}  \otimes \Id_{X'} \otimes \tev_{X^*})(\Id_{X'^*} \otimes f \otimes \Id_{X^{**}})(\tcoev_{X'} \otimes \Id_{X^*} \otimes \phi_X)=\, \psfrag{Y}[Bc][Bc]{\scalebox{.8}{$X'$}} \psfrag{X}[Bc][Bc]{\scalebox{.8}{$X$}} \psfrag{f}[Bc][Bc]{\scalebox{.9}{$f$}} \rsdraw{.45}{.9}{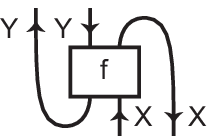}  
$$
and
$$
\beta=(\id_{X^*}\otimes \phi_X^{-1})\, \coev_{X^*} \ev_{X'} =\, \psfrag{Y}[Bc][Bc]{\scalebox{.8}{$X'$}} \psfrag{X}[Bc][Bc]{\scalebox{.8}{$X$}}  \rsdraw{.45}{.9}{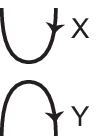} .
$$
Since $X^* \otimes X \in \ideal$, $X'^* \otimes X'  \in \ideal$, and $\tzz$ is a left trace, we have:
$$
\tzz_X(\tr_l^{X^*}(\beta\alpha))=\tzz_{X^* \otimes X}(\beta\alpha)=\tzz_{X'^* \otimes X'}(\alpha\beta)=\tzz_{X'}(\tr_l^{X'^*}(\alpha\beta)).
$$
Now
$$
\tr_l^{X^*}(\beta\alpha)= \psfrag{Y}[Bc][Bc]{\scalebox{.8}{$X'$}} \psfrag{X}[Bc][Bc]{\scalebox{.8}{$X$}} \psfrag{f}[Bc][Bc]{\scalebox{.9}{$f$}} \rsdraw{.45}{.9}{lambi1} =\phi_X^{-1} \bigl(\tr_l^{X'}\!(f)\bigr)^* \phi_X
$$  
and
$$ \tr_l^{X'^*}(\alpha\beta)= \psfrag{Y}[Bc][Bc]{\scalebox{.8}{$X'$}} \psfrag{X}[Bc][Bc]{\scalebox{.8}{$X$}} \psfrag{f}[Bc][Bc]{\scalebox{.9}{$f$}} \rsdraw{.45}{.9}{lambi2} =\tr_r^{X^*}\!(f)\,.
$$

Therefore, \eqref{lambi} is satisfied. 
We deduce Part (b) from Part (a) by using $\cat^\rev$. Let us prove Part (c). Let $f \in \End_\cat(X' \otimes Y \otimes X^*)$, where $X,X' \in \ideal$ and $Y \in \cat$. Set
$$
\alpha=\!\! \psfrag{Y}[Bc][Bc]{\scalebox{.8}{$X'$}} \psfrag{A}[Bc][Bc]{\scalebox{.8}{$Y$}} \psfrag{X}[Bc][Bc]{\scalebox{.8}{$X$}} \psfrag{f}[Bc][Bc]{\scalebox{.9}{$f$}} \rsdraw{.45}{.9}{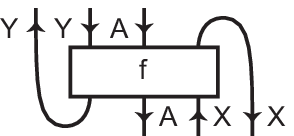}   \quad \text{and} \quad \beta= \psfrag{Y}[Bc][Bc]{\scalebox{.8}{$X'$}} \psfrag{A}[Bc][Bc]{\scalebox{.8}{$Y$}} \psfrag{X}[Bc][Bc]{\scalebox{.8}{$X$}} \psfrag{f}[Bc][Bc]{\scalebox{.9}{$f$}} \rsdraw{.45}{.9}{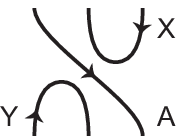} \,.
$$
Since $Y \otimes X^* \otimes X \in \ideal$, $X'^* \otimes X' \otimes Y  \in \ideal$, and $t$ is a trace, we have:
$$
\tzz_{X}(\tr_l^{Y \otimes X^*}(\beta\alpha))=\tzz_{Y \otimes X^* \otimes X}(\beta\alpha)=\tzz_{X'^* \otimes X' \otimes Y}(\alpha\beta)=\tzz_{X'}(\tr_l^{X'^*} \tr_r^Y(\alpha\beta)).
$$
Now
$$
\tr_l^{Y \otimes X^*}(\beta\alpha)= \psfrag{Y}[Bc][Bc]{\scalebox{.8}{$Y$}} \psfrag{W}[Bc][Bc]{\scalebox{.8}{$X'$}} \psfrag{X}[Bc][Bc]{\scalebox{.8}{$X$}} \psfrag{f}[Bc][Bc]{\scalebox{.9}{$f$}} \rsdraw{.45}{.9}{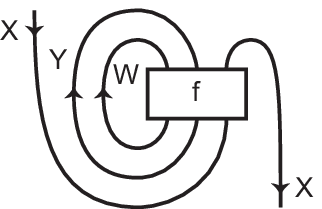}   \quad \text{and} \quad \tr_l^{X'^*} \tr_r^Y(\alpha\beta)= \psfrag{W}[Bc][Bc]{\scalebox{.8}{$X'$}} \psfrag{Y}[Bc][Bc]{\scalebox{.8}{$Y$}} \psfrag{X}[Bc][Bc]{\scalebox{.8}{$X$}} \psfrag{f}[Bc][Bc]{\scalebox{.9}{$f$}} \rsdraw{.45}{.9}{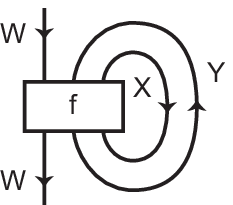} \,.
$$
Therefore, $\tzz$ satisfies \eqref{ambi1}. Likewise, given $g \in \End_\cat(X^* \otimes Y \otimes X')$, by using
$$
\alpha=\!\! \psfrag{Y}[Bc][Bc]{\scalebox{.8}{$X'$}} \psfrag{A}[Bc][Bc]{\scalebox{.8}{$Y$}} \psfrag{X}[Bc][Bc]{\scalebox{.8}{$X$}} \psfrag{f}[Bc][Bc]{\scalebox{.9}{$g$}} \rsdraw{.45}{.9}{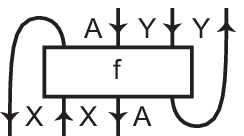}   \quad \text{and} \quad \beta= \psfrag{Y}[Bc][Bc]{\scalebox{.8}{$X'$}} \psfrag{A}[Bc][Bc]{\scalebox{.8}{$Y$}} \psfrag{X}[Bc][Bc]{\scalebox{.8}{$X$}} \psfrag{f}[Bc][Bc]{\scalebox{.9}{$g$}} \rsdraw{.45}{.9}{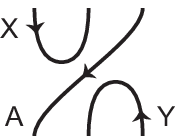} \,,
$$
we obtain that $\tzz$ satisfies \eqref{ambi2}.
\end{proof}

\subsection{Modified dimensions}
If $\tzz$ is a left (resp.\@ right) trace on a left (resp.\@ right) ideal
  $\ideal$ in a pivotal \kt category $\cat$, the
\emph{left} (resp.\@ \emph{right}) \emph{modified dimension} (associated with
$\tzz$) is the function defined on the objects $V\in \ideal$ by
$$\qd_l(V)=\tzz_V(\Id_V) \qquad \left ( \mathrm{resp.\ } \qd_r(V)=\tzz_V(\Id_V) \right ).$$
Note that an immediate consequence of the definition of a left trace (resp.\@ right trace) is that isomorphic objects have equal modified dimensions.

Let $\tzz$ be a trace on an ideal
  $\ideal$ in $\cat$, the
\emph{modified dimension} (associated with
$\tzz$) is the function defined on the objects $V\in \ideal$ by
$$\qd(V)=\tzz_V(\Id_V).$$
Note that if $V \in \ideal$, then $V^*\in\ideal$ and $\qd(V)=\qd(V^*)$ by Lemmas~\ref{lem-dual-ideal} and~\ref{lem-dual-trace}.

\section{Invariants of closed graphs} \label{S:InvClosedGraph}

\subsection{Colored graphs} Let $\cat$ be a pivotal category.  By a
\emph{$\cat$-colored ribbon graph} in an oriented surface $\Sigma$, we mean a
graph embedded in $\Sigma$ whose edges are oriented and colored by objects of
$\cat$ and whose vertices lying in $\operatorname{Int} \Sigma=\Sigma-\partial
\Sigma $ are thickened to coupons colored by morphisms of~$\cat$ (as in
Penrose graphical calculus, see Section~\ref{sec-penrose}).  The edges of a
$\cat$-colored graph do not meet each other and may meet the coupons only at
the bottom and top sides.  The intersection of a $\cat$-colored ribbon graph
in $\Sigma$ with $\partial \Sigma$ is required to be empty or to consist only
of vertices of valency~1.

A $\cat$-colored ribbon graph in $\R^2$ (with counterclockwise orientation) is
called \emph{planar}. A $\cat$-colored ribbon graph in
$S^{2}=\R^2\cup\{\infty\}$ is called \emph{spherical}.

The $\cat$-colored ribbon graphs in $\R\times [0,1]$ (with counterclockwise
orientation) form a category $\Graph_\cat$ as follows: objects of
$\Graph_\cat$ are finite sequences of pairs $(X,\varepsilon)$, where $X\in
\cat$ and $\varepsilon=\pm$.  Morphisms of $\Graph_\cat$ are isotopy classes
of $\cat$-colored ribbon graphs in $\R\times [0,1]$. Composition, identities,
tensor multiplication, left and right duality in $\Graph_\cat$ are defined in
the standard way. In particular,
$$
((V_1,\ep_1),\dots,(V_n,\ep_n))^*=((V_n,-\ep_n),\dots,(V_1,-\ep_1))
$$
and the dual $T^*$ of a $\cat$-colored ribbon graph $T$ in $\R \times [0,1]$
is obtained by rotating~$T$ by $\pi$.
For example, if
\begin{center}
\psfrag{U}[Bl][Bl]{\scalebox{.8}{$U$}}
\psfrag{V}[Br][Br]{\scalebox{.8}{$V$}}
\psfrag{W}[Bl][Bl]{\scalebox{.8}{$W$}}
\psfrag{h}[Bc][Bc]{\scalebox{.9}{$h$}}
$T=$\rsdraw{.45}{.9}{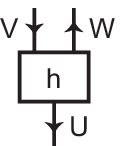} $\co (U,+) \to ((V,+),(W,-))$
\end{center}
where $h \in \Hom_\cat(U,V \otimes W)$, then
\begin{center}
\psfrag{U}[Bl][Bl]{\scalebox{.8}{$U$}}
\psfrag{V}[Bl][Bl]{\scalebox{.8}{$V$}}
\psfrag{W}[Br][Br]{\scalebox{.8}{$W$}}
\psfrag{h}[Bc][Bc]{\scalebox{.9}{$h$}}
$T^*=$\rsdraw{.45}{.9}{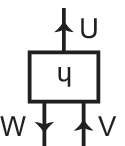} $=$ \psfrag{U}[Br][Br]{\scalebox{.8}{$U$}} \rsdraw{.45}{.9}{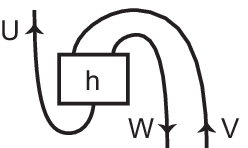} $=$ \psfrag{U}[Bl][Bl]{\scalebox{.8}{$U$}} \rsdraw{.45}{.9}{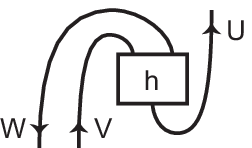} $\co ((W,+),(V,-)) \to (U,-)$.
\end{center}
This makes $\Graph_\cat$ into a pivotal category. By Penrose graphical
calculus, one obtains a (strict) monoidal functor
$$F\colon \Graph_\cat\to\cat.$$
In particular, $F$ sends $(X, \varepsilon)\in\Ob(\Graph_\cat)$ to
$$
X^\varepsilon=\left\{ \begin{array}{ll} X & \text{if $\varepsilon=+$,} \\ X^* & \text{if $\varepsilon=-$.} \end{array} \right.
$$

By a \emph{$\cat$-colored 1-1-ribbon graph}, we mean an endomorphism $T$ in $\Graph_\cat$ of a pair $(X,\varepsilon)$, where $X\in \cat$ and $\varepsilon=\pm$. The pair $(X,\varepsilon)$ is called the \emph{section} of~$T$. The \emph{left and right planar closures} of a $\cat$-colored 1-1-ribbon graph $T$ are the planar $\cat$-colored ribbon graphs (embedded in the interior of $\R\times [0,1]$) defined by:
$$\mathrm{cl}_l(T)=\left\{\begin{array}{ll} \psfrag{T}[Bc][Bc]{\scalebox{.8}{$T$}}
\rsdraw{.45}{.9}{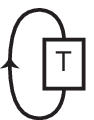} &\text{if $\varepsilon=+$,} \vspace*{.3em} \\ \psfrag{T}[Bc][Bc]{\scalebox{.8}{$T$}} \rsdraw{.45}{.9}{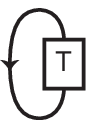} &\text{if $\varepsilon=-$,} \end{array} \right.
 \quad\text{and} \quad \mathrm{cl}_r(T)=\left\{\begin{array}{ll} \psfrag{T}[Bc][Bc]{\scalebox{.8}{$T$}} \rsdraw{.45}{.9}{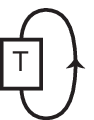} &\text{if $\varepsilon=+$,} \vspace*{.3em} \\ \psfrag{T}[Bc][Bc]{\scalebox{.8}{$T$}} \rsdraw{.45}{.9}{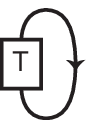} &\text{if $\varepsilon=-$.} \end{array} \right.$$
Note that $F(\mathrm{cl}_l(T))=\tr_l(F(T))$ and $F(\mathrm{cl}_r(T))=\tr_r(F(T))$.

By a \emph{left} (resp.\@ \emph{right}) \emph{cutting presentation} of a planar $\cat$-colored graph $\Gamma$, we mean a  $\cat$-colored  1-1-ribbon graph $T$ such that $\Gamma=\mathrm{cl}_l(T)$ (resp.\@ $\Gamma=\mathrm{cl}_r(T)$).

Considered as spherical $\cat$-colored ribbon graphs, the left and right planar closures of a $\cat$-colored 1-1-ribbon graph $T$ are isotopic, and are called the \emph{spherical closure} of $T$ and denoted $\mathrm{cl}(T)$.
By a \emph{cutting presentation} of a spherical $\cat$-colored graph $\Gamma$, we mean a  $\cat$-colored  1-1-ribbon graph $T$ such that $\Gamma=\mathrm{cl}(T)$.

\subsection{Invariants of admissible graphs}\label{s:inv-G}
Let $\cat$ be a pivotal \kt category and $\ideal$ be  a left (resp.\@ right) ideal of $\cat$. A
$\cat$-colored 1-1-ribbon graph is \emph{left} (resp.\@ \emph{right})
\emph{$\ideal$-admissible} if its section $(V,\varepsilon)$ has the property that
$V^\varepsilon\in\ideal$. By a \emph{left} (resp.\@ \emph{right}) \emph{$\ideal$-admissible planar graph} we mean a planar $\cat$-colored ribbon graphs which admits a left (resp.\@ right) $\ideal$-admissible
left (resp.\@ right) cutting presentation.

If $\tzz$ is a left (resp.\@ right) trace on $\ideal$ and $\Gamma$ is a left (resp.\@ right) $\ideal$-admissible planar graph, we set
$$
    F_\tzz^l(\Gamma)=\tzz_{V^\varepsilon}(F(T)) \qquad \bigl(\text{resp.}\
    F_\tzz^r(\Gamma)=\tzz_{V^\varepsilon}(F(T)) \,\bigr ),
$$
where $T$ is a left (resp.\@ right) $\ideal$-admissible left (resp.\@ right)
cutting presentation of $\Gamma$ with section $(V,\varepsilon)$.

When $\ideal$ is an ideal,  by a \emph{$\ideal$-admissible spherical graph}, we mean spherical $\cat$-colored ribbon graph such that at least one of its edges is colored by an element of~$\ideal$. Note that such a graph admits a cutting presentation with section $(V,+)$ with $V\in\ideal$.

\begin{theorem}\label{thm-inv-graph}
Let $\tzz$ be a left (resp.\@ right) trace on a left (resp.\@ right) ideal $\ideal$ of $\cat$.
Then $F_\tzz^l$ (resp.\@ $F_\tzz^r$) is an isotopy invariant  of left (resp.\@ right)
    $\ideal$-admissible graphs.  Moreover, if $\ideal$ is an ideal and $\tzz$ is a trace on $\ideal$, then $F_\tzz^l=F_\tzz^r$ and this invariant extends to an isotopy invariant of $\ideal$-admissible spherical graphs, denoted $F_\tzz$.
\end{theorem}

When $\ideal=\cat$ and $\tzz$ is the usual left (resp.\@ right) trace of endomorphisms of $\cat$, then   $F_\tzz^l$ (resp.\@ $F_\tzz^r$) is nothing but the usual invariant obtained by Penrose calculus.

Before proving Theorem~\ref{thm-inv-graph}, we introduce some notation.  Remark first that if
$T\colon ((V_1,\ep_1),\ldots,(V_m,\ep_m)) \to ((V_1',\ep'_1),\ldots,(V'_n,\ep'_n))$ is a morphism in $\Graph_\cat$, then
\begin{equation}\label{E:F(T*)FT*}
F(T^{*})=(\psi_{(V_m,\ep_m)}^{-1}\otimes\cdots\otimes\psi_{(V_1,\ep_1)}^{-1})\circ
F(T)^{*}\circ (\psi_{(V'_n,\ep'_n)}\otimes\cdots\otimes\psi_{(V'_1,\ep'_1)})
\end{equation}
where
$\psi_{(V,\ep)}\colon V^{-\ep}\to (V^\ep)^*$ is the isomorphism given by
\begin{equation*}
%  \label{eq:psi}
  \psi_{(V,-1)}=\phi_V\colon V\to V^{**}\text{ and }\psi_{(V,1)}=\Id_{V^{*}}\colon V^{*}\to V^{*}.
\end{equation*}
Let $T$ be an endomorphism of $((V_1,\ep_1),\ldots,(V_m,\ep_m))$ in $\Graph_\cat$. By taking the left closure of the $k$ left most strands of $T$, we define the partial left closure of $T$ denoted by
$$\cl_l^{((V_1,\ep_1),...,(V_{k},\ep_{k}))}(T)\in
\End_{\Graph_\cat}((V_{k+1},\ep_{k+1}),\ldots,(V_{n},\ep_{n})).$$
Similarly,
by taking the right closure of the $k$ right most strands of $T$, we define
the partial right closure of $T$ denoted by
$$\cl_r^{((V_{n-k+1},\ep_{n-k+1}),...,(V_n,\ep_n))}(T)\in
\End_{\Graph_\cat}((V_1,\ep_1),\ldots,(V_{n-k},\ep_{n-k})).$$

Let $\Gamma$ be a spherical $\cat$-colored ribbon graph in $S^2=\R^2\cup\{\infty\}$. By a \emph{cutting path} of
$\Gamma$ we mean an embedded oriented path
$\gamma\colon [0,1]\to S^2$ such that $\gamma(0),\gamma(1)\notin\Gamma$ and whose image does not meet the coupons of $\Gamma$ and
meets any edge of $\Gamma$ transversally.  Let $\gamma$ be a cutting path of $\Gamma$. Consider a tubular  neighborhood
$\Omega\simeq\gamma([0,1])\times]-1,1[$ of the image $\gamma([0,1])$ of $\gamma$. We identify $S^2\setminus\Omega\simeq\R\times[0,1]$ in such a way that
the boundary components $\gamma(]0,1[)\times \{-1\}$ and $\gamma(]0,1[)\times \{1\}$ of $S^2\setminus\Omega$ are sent to  the top boundary $\R\times\{1\}$ and the bottom boundary
$\R\times\{0\}$ of $\R\times[0,1]$,  respectively.
Then we denote by  $cut_\gamma(\Gamma)$ the endomorphism of $\Graph_\cat$ equal to
$\Gamma\setminus \Omega\subset
S^2\setminus\Omega\simeq\R\times[0,1]$.

If $\Gamma$ is a planar $\cat$-colored ribbon graph, considered as a spherical $\cat$-colored graph in $S^2=\R^2\cup\{\infty\}$, and $\gamma$ is a cutting path for $\Gamma$ such that the points $\gamma(1)$ (resp.\@ $\gamma(0)$) and $\infty$
are in the same component of $S^2\setminus\Gamma$,  then $\cl_l(cut_\gamma(\Gamma))$ (resp.\@
$\cl_r(cut_\gamma(\Gamma))$) and $\Gamma$ are isotopic.

\begin{proof}[Proof of Theorem  \ref{thm-inv-graph}]
Let us prove the first statement of 
the theorem.  
Let $\tzz$ be a left (resp.\@ right) trace on a left (resp.\@ right) ideal $\ideal$.  Let $\Gamma$ be an $\ideal$-admissible planar $\cat$-colored
  ribbon graph and $T_0$, $T_1$ be two left cutting presentations of $\Gamma$
  with sections $(V_0,\ep_0)$ and $(V_1,\ep_1)$, respectively. We have to show
  that $\tzz_{V^{\varepsilon_0}}(F(T_0))=\tzz_{V_1^{\varepsilon_1}}(F(T_1))$.
  The two edges $e_0$ and $e_1$ of $\Gamma$ that are cut to form $T_0$ and
  $T_1$, respectively, are in the boundary of the unbounded component $C$ of
  $\R^2\setminus\Gamma$.  Let $\gamma_0$ and $\gamma_1$ be two disjoint
  cutting paths located in a neighborhoods of $e_0$ and $e_1$, respectively,
  such that $T_i=cut_{\gamma_i}(\Gamma)$ and $\gamma_i(1)\in C$ for $i=0,1$.
  Choose an embedded path $\gamma_2\colon [0,1]\to C\setminus\bigl
  (\gamma_0(]0,1[)\cup\gamma_1(]0,1[)\bigr)$ such that
  $\gamma_2(0)=\gamma_0(1)$ and $\gamma_2(1)=\gamma_1(1)$.  Define $\gamma$ as
  the concatenation of paths $\gamma=\bar\gamma_1\gamma_2\gamma_0$ where
  $\bar\gamma_1(t)=\gamma_1(1-t)$. Set
  $$T=cut_\gamma(\Gamma)\in\End_{\Graph_\cat}((V_0,\ep_0),(V_1,-\ep_1)).$$
  When $\ep_0=\ep_1=1$, the construction of $T$ can be schematically
  depicted~as:
  $$
  \psfrag{k}[Bc][Bc]{\scalebox{.9}{$\gamma_1$}}
  \psfrag{e}[Bl][Bl]{\scalebox{.9}{$e_0$}}
  \psfrag{c}[Br][Br]{\scalebox{.9}{$e_1$}}
  \psfrag{b}[Bc][Bc]{\scalebox{.9}{$\gamma_2$}}
  \psfrag{l}[Bc][Bc]{\scalebox{.9}{$\gamma_0$}}
  \psfrag{T}[Bc][Bc]{\scalebox{.9}{$T$}}
  \rsdraw{.45}{.9}{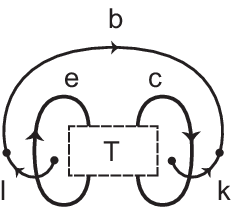}\,.
  $$
By construction, we have: $T_0=\cl_r^{(V_1,-\ep_1)}(T)$ and $T_1=\cl_r^{(V_0,-\ep_0)}(T^{*})$. Then, setting $$f=(\Id_{V_0^{\ep_0}}\otimes\psi_{(V_1,\ep_1)})  F(T)  (\Id_{V_0^{\ep_0}}\otimes
  \psi_{(V_1,\ep_1)}^{-1}) \in\End_\cat(V_0^{\ep_0}\otimes (V_1^{\ep_1})^{*}),$$ we have:
\begin{align*}
\tzz_{V_0^{\ep_0}}(F(T_0))&=
  \tzz_{V_0^{\ep_0}}\bigl(F(\cl_r^{(V_1,-\ep_1)}(T))\bigr)\\
  &=\tzz_{V_0^{\ep_0}}\bigl(\tr_r^{V_1^{-\ep_1}}\!(F(T))\bigr)\\
  & =\tzz_{V_0^{\ep_0}}\bigl(\tr_r^{V_1^{-\ep_1}}\!\bigl(
    (\Id_{V_0^{\ep_0}}\otimes\psi_{(V_1,\ep_1)}^{-1})
    f(\Id_{V_0^{\ep_0}}\otimes\psi_{(V_1,\ep_1)})\bigr)
  \bigr)\\
  &=\tzz_{V_0^{\ep_0}}\bigl(\tr_r^{(V_1^{\ep_1})^*}(f)\bigr)
\end{align*}
and
\begin{align*}
\tzz_{V_1^{\ep_1}}(F(T_1))&=
  \tzz_{V_1^{\ep_1}}\bigl(F(\cl_r^{(V_0,-\ep_0)}(T^{*}))\bigr)
  =\tzz_{V_1^{\ep_1}}\bigl(\tr_r^{V_0^{-\ep_0}}\!(F(T^{*}))\bigr).
\end{align*}
Now
\begin{align*}
F(T^{*}) &= (\psi_{(V_1,{-\ep_1})}^{-1}\otimes
  \psi_{(V_0,{\ep_0})}^{-1}) F(T)^{*}
  (\psi_{(V_1,{-\ep_1})}\otimes \psi_{(V_0,{\ep_0})})\\
  &=(\phi_{V_1^{\ep_1}}^{-1}\otimes \psi_{(V_0,{\ep_0})}^{-1})
  f^{*} (\phi_{V_1^{\ep_1}}\otimes \psi_{(V_0,{\ep_0})})
\end{align*}
 by using \eqref{E:F(T*)FT*} and the fact that $\psi_{(V,\ep)}^{*}=\psi_{(V^{*},\ep)}^{-1}$. Therefore
  \begin{align*}
  \tr_r^{V_0^{-\ep_0}}\!(F(T^{*})) &= \phi_{V_1^{\ep_1}}^{-1}
  \tr_r^{V_0^{-\ep_0}}\!\bigl((\Id\otimes \psi_{V_0,{\ep_0}}^{-1})
  f^{*} (\Id\otimes \psi_{V_0,{\ep_0}})\bigr)
  \phi_{V_1^{\ep_1}}\\
  &= \phi_{V_1^{\ep_1}}^{-1} \tr_r^{(V_0^{\ep_0})^*}\!(f^{*})
  \phi_{V_1^{\ep_1}}\\
  &=\phi_{V_1^{\ep_1}}^{-1}\bigl(
  \tr_l^{V_0^{\ep_0}}\!(f)\bigr)^{*} \phi_{V_1^{\ep_1}}
  \end{align*}
and so
\begin{align*}
\tzz_{V_1^{\ep_1}}(F(T_1))&=\tzz_{V_1^{\ep_1}}\bigl(\tr_r^{V_0^{-\ep_0}}\!(F(T^{*}))\bigr)=\tzz_{V_1^{\ep_1}} \bigl(\phi_{V_1^{\ep_1}}^{-1}\bigl(
  \tr_l^{V_0^{\ep_0}}\!(f)\bigr)^{*} \phi_{V_1^{\ep_1}}\bigr).
\end{align*}
Finally, 
by Lemma~\ref{lem-trace-on-ideal-are-ambi}(a), we have 
  $$\tzz_{V_1^{\ep_1}}(F(T_1))=\tzz_{V_1^{\ep_1}} \bigl(\phi_{V_1^{\ep_1}}^{-1}\bigl(
  \tr_l^{V_0^{\ep_0}}\!(f)\bigr)^{*} \phi_{V_1^{\ep_1}}\bigr)
  = \tzz_{V_0^{\ep_0}}\bigl(\tr_r^{(V_1^{\ep_1})^*}(f)\bigr) =\tzz_{V_0^{\ep_0}}(F(T_0)).$$
Thus, $F_\tzz^l$ is an isotopy invariant of left  $\ideal$-admissible graphs.  Then using
  $\cat^\rev$ it follows that $F_\tzz^r$ is an isotopy invariant of right
    $\ideal$-admissible graphs when $\tzz$ is a left trace.
This concludes the proof of the first statement in Theorem \ref{thm-inv-graph}.

We now prove the second statement of 
the theorem.  Let $\tzz$ be a trace on an ideal~$\ideal$.
  Let $\Gamma$ be an $\ideal$-admissible spherical $\cat$-colored ribbon graph
  and $T_0$, $T_1$ be two cutting presentations of $\Gamma$ with sections
  $(V,\ep)$ and $(V',\ep')$, respectively. We have to show that
  $\tzz_{{V'}^{\varepsilon'}}(F(T_1))=\tzz_{V^{\varepsilon}}(F(T_0))$.
  Let $e_0$ and $e_1$ be the edges of $\Gamma$ that are cut to form $T_0$ and
  $T_1$, respectively. Let $\gamma_0$ and $\gamma_1$ be two disjoint cutting
  paths located in a neighborhoods of $e_0$ and $e_1$, respectively, such that
  $T_i=cut_{\gamma_i}(\Gamma)$ for $i=0,1$.
  Choose an embedded path $\gamma_2\colon [0,1]\to
  S^2\setminus(\gamma_0([0,1[)\cup\gamma_1([0,1[))$ such that
  $\gamma_2(0)=\gamma_0(1)$, $\gamma_2(1)=\gamma_1(1)$, and whose image does
  not meet the coupons of $\Gamma$ and meets any edge of $\Gamma$
  transversally. Define $\gamma$ as the concatenation of paths
  $\gamma=\bar\gamma_1\gamma_2\gamma_0$, where
  $\bar\gamma_1(t)=\gamma_1(1-t)$.  Set
    $$T=cut_\gamma(\Gamma)\in\End_{\Graph_\cat}((V_0,\ep_0),(V_1,\ep_1),\ldots,(V_n,\ep_n))$$
where $(V_0,\ep_0)=(V,\ep)$, $(V_n,\ep_n)=(V',\ep')$, and $V_1, \dots,V_n$ are the colors of the edges met by $\gamma_2$.
By construction,
$$T_0=\cl_r^{((V_1,\ep_1),\ldots,(V_n,\ep_n))}(T) \;\text{ and } \; T_1=\cl_r^{((V_{n-1},-\ep_{n-1}),\ldots,(V_0,-\ep_0))}(T^{*}).$$
Let $Y=V_1^{\ep_1}\otimes\cdots\otimes V_{n-1}^{\ep_{n-1}}$ and set $$f=(\Id_{V_0^{\ep_0}\otimes Y}\otimes\psi_{V_n,\ep_n}) F(T) (\Id_{V_0^{\ep_0}\otimes Y}\otimes
  \psi_{V_n,\ep_n}^{-1}) \in\End_\cat(V_0^{\ep_0}\otimes Y \otimes (V_n^{\ep_n})^{*}).$$  As in the first part of the theorem, one can show that
$$
\tzz_{V_0^{\ep_0}}(F(T_0))=\tzz_{V_0^{\ep_0}}\bigl(\tr_r^{Y\otimes (V_n^{\ep_n})^*}\!(f)\bigr)
$$
and
$$ \tzz_{V_n^{\ep_n}}(F(T_1))=\tzz_{V_n^{\ep_n}}\bigl(\phi_{V_n^{\ep_n}}^{-1}\bigl(\tr_l^{V_0^{\ep_0}\otimes Y}\!(f)\bigr)^{*}\phi_{V_n^{\ep_n}}\bigr).
$$
Thus, the second statement of theorem follows from these formulas and Equation \eqref{ambi1}.  
\end{proof}

\subsection{Invariants of spherical graphs from one-sided traces}\label{s:inv-G2}
Let $\cat$ be a pivotal \kt category.
By a \emph{$\A$-colored graph}, where $\A$ is a class of object of $\cat$, we mean a
$\cat$-colored graph whose edges are colored by elements of $\A$.

 Let $\tzz$ be a left (resp.\@ right) trace on a left (resp.\@ right)
    ideal $\ideal$.  Set
   \begin{equation}\label{E:DefA}
    \A=\{V\in\ideal\cap\ideal^{*}\,|\, \tzz_V=\tzz{\vee}_{V}\},
  \end{equation}
where $\tzz{\vee}$ is defined in Lemma~\ref{lem-dual-trace}(c).
    For any spherical $\A$-colored ribbon
    graph $\Gamma$, set
    $$
    F_\tzz^{a}(\Gamma)=\tzz_{V^\varepsilon}(F(T)) ,
    $$
    where $T$ is any cutting presentation of $\Gamma$ with section
    $(V,\varepsilon)$.

\begin{theorem}\label{thm-inv-graph2}
$F_\tzz^a$ is an isotopy invariant of spherical $\A$-colored ribbon
    graphs.
\end{theorem}
Note that if $\tzz$ is a trace on an ideal $\ideal$, then $\A=\ideal$ (by Lemmas~\ref{lem-dual-ideal} and~\ref{lem-dual-trace}), and the invariants $F_\tzz$ and $F_\tzz^{a}$ coincide on  spherical $\ideal$-colored ribbon graphs.

\begin{proof}
We prove Theorem  \ref{thm-inv-graph2} in the case when $\tzz$ is a left trace.  Then the case when $\tzz$ is a right trace can be deduced using $\cat^\rev$.  Let $\Gamma$
  be a spherical $\A$-colored ribbon graph and let $T$ be a cutting presentation of $\Gamma$.
  Then $\Gamma'=\cl_l(T)$ is a planar $\A$-colored  ribbon graph whose isotopy class
  in $S^2=\R^2\cup\{\infty\}$ is the same as the isotopy class of $\Gamma$.
  By Theorem \ref{thm-inv-graph},
$F^l_\tzz(\Gamma')$ does not depend on the left cutting
   presentations of $\Gamma'$.
     Let $e$ be the edge of $\Gamma'$ that is cut to form $T$.  Let $C$ and~$C'$
  be the two connected components of $S^2\setminus\Gamma$ located on the two
  sides of $e$, where~$C$ is the distinguished component containing $\infty$.
Then $\Gamma'=\cl_l(T)$ and
  $\Gamma''=\cl_r(T)=\cl_l(T^{*})$ are isotopic in $S^2$.  However $C'$ is the
  distinguished component of $\Gamma''$ containing $\infty$.
  Hence, $\Gamma''$ is obtained from $\Gamma'$ by a move that consist in
  pushing the point $\infty$ across an edge.  The class of $\Gamma'$ modulo
  these moves clearly depends only of $\Gamma$ and thus it is enough to show
  that $F^l_\tzz(\Gamma')$ is invariant by this move.  This is true because
  $$F^l_t(\Gamma')
  =\tzz_{V^\ep}(F(T))=\tzz{\vee}_{V^{\ep}}(F(T))
   =\tzz_{(V^{\ep})^*}(F(T)^*)
  =\tzz_{V^{-\ep}}(F(T^*))
  =F^l_t(\Gamma'')$$
  where the second equality is due to
  $\tzz{\vee}=\tzz$ on $\A$ and the fourth equality follows from Equations \eqref{defltrace} and \eqref{E:F(T*)FT*}.
\end{proof}

The invariant $F_t^a$ of Theorem~\ref{thm-inv-graph2} is a generalization of the analogous invariant defined from ribbon categories in \cite{GPT1}.
Moreover,   $F_\tzz^a$ produces the data of a trivalent-ambidextrous
pair as required in \cite{GPT2}.
More precisely, let $\B$ be
a class of simple objects of $\cat$ such that $V^*\in\B$ for all $V \in\B$. Denote by $\tet_{\B}$ the class of
connected trivalent spherical $\B$-colored ribbon graphs (a ribbon graph is
\emph{trivalent} if all its coupons are adjacent to 3 half-edges). Let $\qd\colon \B\to
\kk$ be a map such that, for all $V,V' \in \B$,
  \begin{enumerate}
    \renewcommand{\labelenumi}{{\rm (\roman{enumi})}}
  \item  $\qd(V)=\qd(V^*)$,
  \item $\qd(V)=\qd(V')$ if $V$ is isomorphic to
$V'$.
  \end{enumerate}
  For any 1-1-ribbon graph $Q$ with section $(V,\varepsilon)$
  where $V$ is simple, we let $\langle Q\rangle =\langle
  F(Q)\rangle_{V^\varepsilon} \in \kk$, that is, $F(Q)= \langle Q\rangle \,
  \Id_{V^\varepsilon}$ (see Section~\ref{sphesphe}).  Using
  this notation, the pair $(\B,\qd)$ is \emph{trivalent-ambidextrous} if for
  any $\Gamma\in \tet_{\B}$ and for any two cutting presentations $T, T'$ of
  $\Gamma$ with sections $(V,\varepsilon)$ and $(V',\varepsilon')$,
  we have:
  $$\qd(V)\ang{T} \, =\qd(V')\ang{T'}.$$
For a
trivalent-ambidextrous pair $(\B,\qd)$, we define a function $G_{(\B,d)}\colon
\tet_{\B} \rightarrow \kk$ by
\begin{equation*}
  G_{(\B,d)}(\Gamma)={\qd}(V)\langle T\rangle
\end{equation*}
where $T$ is any cutting presentation of $\Gamma$ with section
$(V,\varepsilon)$ with $V\in \B$.  The definition of a trivalent-ambidextrous
pair implies that $G_{(\B,d)}$ is well-defined.

Let us explain how to  produce a trivalent-ambidextrous
pairs from traces on an ideals and how the invariants derived from such data are related. Let $\tzz$ be a left (resp.\@ right) trace on a left (resp.\@ right) ideal $\ideal$ of $\cat$ and $\A$ as above. Denote by $\qd$ the left (resp.\@ right) modified dimension on $\ideal$ associated with $\tzz$. Set
 $$
  \B=\{V\in\ideal\cap\ideal^{*}\,|\, V\text{ is simple and
  }\qd(V)=\qd(V^{*})\}.
  $$
\begin{corollary}
The pair $(\B,\qd)$ is
  trivalent-ambidextrous, $\B \subset \A$ (and so any $\Gamma \in \tet_{\B}$ is $\A$-colored), and $G_{(\B,\qd)}(\Gamma)=F_\tzz^a(\Gamma)$ for all $\Gamma\in \tet_{\B}$.
\end{corollary}
\begin{proof}
Clearly $(\A,\qd)$ is a
  trivalent-ambidextrous pair since isomorphic objects have equal modified dimension. Now let $f\in\End_\cat(V)$ where $V \in \A$. Since $V$ is simple, we have: $f=\langle f \rangle_V \id_V$ and $\langle f^* \rangle_{V^*}=\langle f \rangle_V$. Therefore $$\tzz_V(f)=\langle f \rangle_V \qd(V)=\langle f^* \rangle_{V^*}\qd(V^*)=\tzz_{V^*}(f^*)=\tzz\vee_V(f).$$
Finally, let $\Gamma\in\tet_{\B}$ and $T$ be any cutting presentation of $\Gamma$ with section
    $(V,\varepsilon)$. Since $V$ is simple, $F(T)= \langle T\rangle \, \Id_{V^\varepsilon}$.
    Then \begin{align*}
    F_\tzz^{a}(\Gamma)&=\tzz_{V^\varepsilon}(F(T))=\tzz_{V^\varepsilon}(\langle T\rangle \, \Id_{V^\varepsilon})\\
    &=\langle T\rangle \,\tzz_{V^\varepsilon}( \Id_{V^\varepsilon})
    = \langle T\rangle \,\qd(V^\varepsilon)=\langle T\rangle \,\qd(V) = G_{(\B,\qd)}(\Gamma).
   \end{align*}
Therefore, $(F_t^a)_{|\tet_{\B}}=G_{(\B,\qd)}$.
\end{proof}

\section{Traces and ideals from classes of objects} \label{S:AmbiTrace}
In this section we give a systematic approach to defining traces on ideals.   In particular, we introduce the notions of a one-sided ambidextrous trace on a class of objects and of a spherical trace on a class of objects.   Then we show that such traces are in one to one correspondence with the traces of Section \ref{sect-ideaux}.

\subsection{Ideals generated by a class of objects}
Let $\cat$ be a monoidal category. Given a class $\oo \subset \cat$, set
\begin{align*}
& \ideal^l_\oo=\bigl\{ U \in \cat \, \big | \, \text{$U$ is a retract of $Y \otimes X$ for some $X \in \oo$ and $Y\in\cat$}  \bigr \},\\
& \ideal^r_\oo=\bigl\{ U \in \cat \, \big | \, \text{$U$ is a retract of $X \otimes Z$ for some $X \in \oo$ and $Z\in\cat$}  \bigr \},\\
& \ideal_\oo=\bigl\{ U \in \cat \, \big | \, \text{$U$ is a retract of $Y \otimes X \otimes Z$ for some $X \in \oo$ and $Y,Z\in\cat$}  \bigr \}.
\end{align*}
Then  $\ideal^l_\oo$  is the smallest left ideal of $\cat$ containing $\oo$,  $\ideal^r_\oo$  is the smallest right ideal of $\cat$ containing $\oo$, and $\ideal_\oo$  is the smallest ideal of $\cat$ containing $\oo$.

In the case where $\oo=\{V\}$ for $V \in \cat$, we denote $\ideal^l_{\{V\}}$, $\ideal^r_{\{V\}}$, and $\ideal_{\{V\}}$ by  $\ideal^l_V$, $\ideal^r_V$, and $\ideal_V$ respectively.

\begin{lemma}\label{P:ideal4} Let $\cat$ be a pivotal category, $\oo \subset \cat$, and  $U \in \cat$.
  \begin{enumerate}
\renewcommand{\labelenumi}{{\rm (\alph{enumi})}}
\item  $U \in \ideal^l_\oo$ if and only if there exist $X \in \oo$ and $f\in \End_\cat(U \otimes X^*)$ such
  that $\tr_r^{X^*}\!(f)=\Id_U$.
\item $U \in \ideal^r_\oo$ if and only if there exist $Y \in \oo$ and $f\in \End_\cat(Y^*\otimes U)$ such
  that $\tr_l^{Y^*}\!(f)=\Id_U$.
\item Let $\oo'=\{X^{*}\, | \, X\in\oo\}$.  Then
  $(\ideal^l_\oo)^{*}=\ideal^r_{\oo'}$, $(\ideal^r_\oo)^{*}=\ideal^l_{\oo'}$,
  and $\ideal_\oo=\ideal_{\oo'}$.
\end{enumerate}
\end{lemma}
\begin{proof}
Let us prove Part (a). Assume $U \in \ideal^l_\oo$, that is,  there exist $X \in \oo$, $Y\in\cat$, $p\colon Y \otimes X \to U$, and $q \colon U \to Y \otimes X$ such that $pq=\Id_U$. Set
$$
f=(p \otimes \Id_{X^*})(\Id_Y \otimes \tcoev_X \ev_X)(q \otimes \Id_{X^*}) \colon U \otimes X^* \to U \otimes X^*.
$$
Then $\tr_r^{X^*}\!(f)=pq=\Id_U$. Conversely, assume that there exist $X \in \oo$ and $f\in \End_\cat(U \otimes X^*)$ such
that $\tr_r^{X^*}\!(f)=\Id_U$. Note that $U \otimes X^* \otimes X \in \ideal^l_\oo$ and set
\begin{align*}
&p=\Id_U \otimes \ev_X \colon U \otimes X^* \otimes X \to U,\\
&q=(f \otimes \Id_X)(\Id_U \otimes \tcoev_X) \colon   U \to U \otimes X^* \otimes X.
\end{align*}
Then $pq=\tr_r^{X^*}\!(f)=\Id_U$.

One deduces Part (b) from Part (a) using $\cat^\rev$. The first two identities of Part~(c) follow from the observations in the proof of Lemma \ref{L:id*}.  Finally  $\ideal_\oo=(\ideal_\oo)^{*}=\ideal_{\oo'}$ by Lemma \ref{L:id*}.
\end{proof}

\subsection{Ambidextrous traces on a class of objects}
Let $\cat$ be a pivotal \kt category and $\oo \subset \cat$.  Denote by $\phi=\{\phi_X\colon X\to X^{**}\}_{X \in \cat}$ the pivotal structure of $\cat$ (see Section~\ref{sect-picotal-cat}). Let
$t=\{t_X\colon \End_\cat(X)\rightarrow \kk\}_{X\in\oo}$
be a family of \kt linear forms.

We say that the family $t$ is a \emph{left ambidextrous trace on $\oo$} if \eqref{lambi} is satisfied for all $X,X' \in \oo$ and $f \in \End_\cat(X' \otimes X^*)$, that is,
\begin{equation*}
t_X\left (\, \psfrag{Y}[Bc][Bc]{\scalebox{.8}{$X'$}} \psfrag{X}[Bc][Bc]{\scalebox{.8}{$X$}} \psfrag{f}[Bc][Bc]{\scalebox{.9}{$f$}} \rsdraw{.45}{.9}{lambi1}  \, \right ) = t_{X'} \left (\, \psfrag{Y}[Bc][Bc]{\scalebox{.8}{$X'$}} \psfrag{X}[Bc][Bc]{\scalebox{.8}{$X$}} \psfrag{f}[Bc][Bc]{\scalebox{.9}{$f$}} \rsdraw{.45}{.9}{lambi2}  \right ).
\end{equation*}

We say that the family $t$ is a  \emph{right ambidextrous trace on $\oo$} if \eqref{rambi} is satisfied for all $X,X' \in \oo$ and $g \in \End_\cat(X^* \otimes X')$, that is, 
\begin{equation*}
t_X\left ( \, \psfrag{Y}[Bc][Bc]{\scalebox{.8}{$X'$}} \psfrag{X}[Bc][Bc]{\scalebox{.8}{$X$}} \psfrag{f}[Bc][Bc]{\scalebox{.9}{$g$}} \rsdraw{.45}{.9}{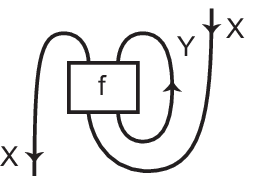}  \, \right ) = t_{X'} \left ( \psfrag{Y}[Bc][Bc]{\scalebox{.8}{$X'$}} \psfrag{X}[Bc][Bc]{\scalebox{.8}{$X$}} \psfrag{f}[Bc][Bc]{\scalebox{.9}{$g$}} \rsdraw{.45}{.9}{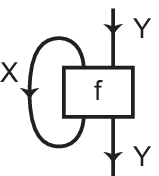} \,\right ).
\end{equation*}

For example, the left (resp.\@ right) trace of endomorphisms in $\cat$ (see
Section~\ref{sec-traces}) is a left (resp.\@ right)  ambidextrous trace on
$\Ob(\cat)$.

\begin{rem}
  If $\cat$ is a ribbon category and $V\in\cat$, we recover the definition of an
  ambidextrous trace on $\oo=\{V\}$ given in \cite{GKP1}. Indeed, in that case, the notions
  of left and right ambidextrous traces become equivalent: the $\kk$-linear
  map
  $$\gamma\colon \End_\cat(V\otimes V)\to\End_\cat(V^{*}\otimes
  V), \quad \text{defined by } g \mapsto
  \gamma(g)=\psfrag{V}[Bc][Bc]{\scalebox{.8}{$V$}}
  \psfrag{g}[Bc][Bc]{\scalebox{1}{$g$}} \rsdraw{.45}{.9}{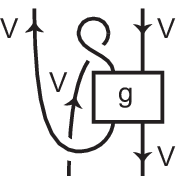}\;,
  $$
  is an isomorphism and, given a map $t_V \colon \End_\cat(V)\rightarrow \kk$,
  the morphism $f=\gamma(g)$ satisfies \eqref{rambi} (for $t_V$) if and only
  if $t_V(\tr_r^V(g))=t_V(\tr_l^V(g))$.
\end{rem}

By Lemma~\ref{lem-trace-on-ideal-are-ambi}, any left (resp.\@ right) trace on a left (resp.\@ right) ideal
$\ideal$ in $\cat$ is a left (resp.\@ right) ambidextrous trace on $\ideal$ (and in particular on any $\oo \subset \ideal$).
The following theorem states that one-sided ambidextrous traces on a class of objects bijectively correspond to one-sided  traces on the one-sided  ideal generated by the class.
\begin{theorem}\label{thm:sided-traces}
Let $\cat$ be a pivotal \kt category. If $t$ is a left (resp.\@ right) ambidextrous trace on a class $\oo$
  of objects of~$\cat$, then there exists a unique left (resp.\@ right) trace
  $\tzz$ on $\ideal_\oo^l$ (resp.\@ $\ideal_\oo^r$) such that
  $\tzz_{|\oo}=t$.
\end{theorem}
We prove Theorem~\ref{thm:sided-traces} in Section~\ref{proof-thm-sided-trace}.

\begin{corollary}\label{C:sided-traceIambi}
  Let $\cat$ be a pivotal \kt category,  $\ideal$ be a left (resp.\@ right) ideal in $\cat$, and
 $t=\{t_X\colon \End_\cat(X)\rightarrow \kk\}_{X\in\ideal}$ be a family of
  \kt linear maps. Then $t$ is a
  left (resp.\@ right) trace on $\ideal$ if and only if $t$ is a left (resp.\@
  right) ambidextrous trace on~$\ideal$.
\end{corollary}
\begin{proof}
This is a direct consequence of Theorem~\ref{thm:sided-traces} and the fact that if $\ideal$ is a left ideal (resp.\@ right ideal), then $\ideal^l_\ideal=\ideal$ (resp.\@ $\ideal^r_\ideal=\ideal$).
\end{proof}

\subsection{Spherical traces  on a class of objects} Let $\cat$ be a  pivotal \kt category and  $t=\{t_X\colon \End_\cat(X)\rightarrow \kk\}_{X\in\ideal}$ be a family of
  \kt linear maps.

\begin{lemma}\label{lem-ambi1}
The following assertions are equivalent:
  \begin{enumerate}
\renewcommand{\labelenumi}{{\rm (\roman{enumi})}}
\item The family $t$ satisfies \eqref{ambi1} for all $X,X' \in \oo$, $Y \in \cat$, and $f \in \End_\cat(X' \otimes Y \otimes X^*)$, that is, 
\begin{equation*}
    t_X\left (\, \psfrag{Y}[Bc][Bc]{\scalebox{.8}{$Y$}} \psfrag{W}[Bc][Bc]{\scalebox{.8}{$X'$}} \psfrag{X}[Bc][Bc]{\scalebox{.8}{$X$}} \psfrag{f}[Bc][Bc]{\scalebox{.9}{$f$}} \rsdraw{.45}{.9}{lambi1b}  \, \right ) = t_{X'} \left (\, \psfrag{W}[Bc][Bc]{\scalebox{.8}{$X'$}} \psfrag{Y}[Bc][Bc]{\scalebox{.8}{$Y$}} \psfrag{X}[Bc][Bc]{\scalebox{.8}{$X$}} \psfrag{f}[Bc][Bc]{\scalebox{.9}{$f$}} \rsdraw{.45}{.9}{lambi2b}  \right ).
\end{equation*}
\item The family $t$ satisfies  \eqref{ambi2}  for all $X,X' \in \oo$, $Y \in \cat$, and $g\in \End_\cat(X^*\otimes Y \otimes X')$, that is,
\begin{equation*}
    t_X\left (\, \psfrag{Y}[Bc][Bc]{\scalebox{.8}{$Y$}} \psfrag{W}[Bc][Bc]{\scalebox{.8}{$X'$}} \psfrag{X}[Bc][Bc]{\scalebox{.8}{$X$}} \psfrag{f}[Bc][Bc]{\scalebox{.9}{$g$}} \rsdraw{.45}{.9}{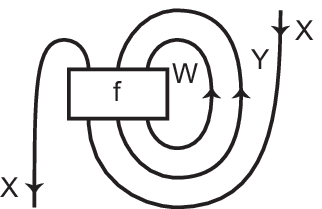}  \, \right ) = t_{X'} \left (\, \psfrag{W}[Bc][Bc]{\scalebox{.8}{$X'$}} \psfrag{Y}[Bc][Bc]{\scalebox{.8}{$Y$}} \psfrag{X}[Bc][Bc]{\scalebox{.8}{$X$}} \psfrag{f}[Bc][Bc]{\scalebox{.9}{$g$}} \rsdraw{.45}{.9}{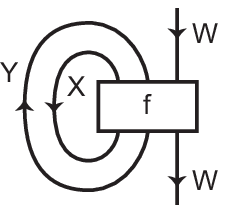}  \right ).
\end{equation*}
\end{enumerate}
\end{lemma}
Lemma~\ref{lem-ambi1} is proved in Section~\ref{proof-thm-trace}.

We say that the family $t$ is a \emph{spherical trace} on $\oo$ if it satisfies the equivalent assertions of Lemma~~\ref{lem-ambi1}.
When $\cat$ is spherical (see Section~\ref{sec-spherical}), the
trace of endomorphisms in $\cat$ is a spherical trace on $\Ob(\cat)$.

By Lemma~\ref{lem-trace-on-ideal-are-ambi}, any trace on an ideal
$\ideal$ in $\cat$ is an ambidextrous trace on $\ideal$ (and in particular on any $\oo \subset \ideal$).
The following theorem states that spherical traces on a class of objects bijectively correspond to traces on the ideal generated by the class.
\begin{theorem}\label{thm:traces}
If $t$ is a spherical trace on a class $\oo$ of objects of~$\cat$,
then there exists a unique trace $\tzz$ on $\ideal_\oo$ such that
$\tzz_{|\oo}=t$.
\end{theorem}
We prove Theorem~\ref{thm:traces} in Section~\ref{proof-thm-trace}.

\begin{corollary}
  Let $\ideal$ be an ideal in $\cat$ and
 $t=\{t_X\colon \End_\cat(X)\rightarrow \kk\}_{X\in\ideal}$ be a family of
  \kt linear maps. Then $t$ is a
  trace on $\ideal$ if and only if $t$ is a spherical trace on $\ideal$.
\end{corollary}
\begin{proof}
This is a direct consequence of Theorem~\ref{thm:traces} and the fact that if $\ideal$ is an ideal, then  $\ideal_\ideal=\ideal$.
\end{proof}

A spherical trace on $\oo$ is in particular  a left and a right ambidextrous trace on
$\oo$. But the converse is not true in general (for example, the left trace of endomorphisms in $\cat$ is both a left and right ambidextrous trace on $\oo=\{\unit\}$ but is not a spherical trace on $\{\unit\}$ unless $\cat$ is spherical). Nevertheless the converse is true when $\cat$ is ribbon:

\begin{corollary}
Assume $\cat$ is ribbon. Let $\oo$ be a class of objects of $\cat$. If $t$ is both a left and right ambidextrous trace on $\oo$, then $t$ is a spherical trace on $\oo$.
\end{corollary}
\begin{proof}
Since $\cat$ is braided, $\ideal_\oo^l=\ideal_\oo^r=\ideal_\oo$.
By Theorem~\ref{thm:sided-traces}(a), $t$ extends (uniquely) to a right trace $\tzz$ on $\ideal_\oo$. Let $X,X' \in \ideal_\oo$ and $f \in \End_\cat(X' \otimes X^*)$. We have:
$$
\psfrag{Y}[Bc][Bc]{\scalebox{.8}{$X'$}} \psfrag{X}[Bc][Bc]{\scalebox{.8}{$X$}} \psfrag{f}[Bc][Bc]{\scalebox{.9}{$f$}}
\tzz\!\left(\,\rsdraw{.45}{.9}{lambi1} \right)
= \tzz\!\left(\rsdraw{.45}{.9}{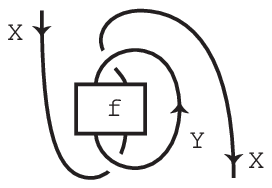}\, \right)
= \tzz\!\left(\rsdraw{.45}{.9}{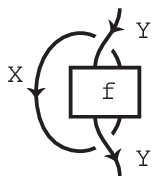} \; \right)
= \tzz\!\left(\;\rsdraw{.45}{.9}{lambi2} \right).
$$
The first and third equalities follow from the fact that $\cat$ is ribbon. Since $\tzz$ is a right ambidextrous trace on $\ideal_\oo$ by Theorem~\ref{thm:sided-traces}(b), then the second equality follows from \eqref{rambi} with $g=\tau_{X',X^*}f\tau_{X',X^*}^{-1}$, where $\tau$ is the braiding of $\cat$. Thus $\tzz$ is a left ambidextrous trace on $\ideal_\oo$, and so a left trace on $\ideal_\oo$ by Corollary~\ref{C:sided-traceIambi}. Hence $\tzz$ is a trace on $\ideal_\oo$, and so $t=\tzz_{|\oo}$ is a spherical trace on $\oo$ by Theorem\ref{thm:traces}(b).
\end{proof}

\subsection{Traces from ambidextrous objects}\label{sect-ambidextrous-object}
Let $\cat$ be a monoidal \kt category. If $V$ is a simple object of $\cat$, we
denote by $\brk{\,}_V \colon \End_\cat(V) \to \kk$ the inverse of the \kt
linear isomorphism $\kk \to \End_\cat(V)$ defined by $k \mapsto k\, \Id_V$.

We say that an object $V$ of $\cat$ is \emph{left ambidextrous} (resp.\@
\emph{right ambidextrous}, resp.\@ \emph{spherical}) if it is simple and
the \kt linear form $\brk{\,}_V$ is a left ambidextrous trace (resp.\@ right
ambidextrous trace, resp.\@ spherical trace) on $\{V\}$. By Theorems~\ref{thm:sided-traces} and \ref{thm:traces}, such an object gives rise to a left trace (resp.\@ right trace, resp.\@ trace) on $\ideal_V^l$ (resp.\@ $\ideal_V^r$, resp.\@ $\ideal_V$).

The following proposition provides useful characterizations of ambidextrous objects:
\begin{proposition} Let $\cat$ be a pivotal \kt category, with pivotal structure
$\phi$, and let $V$ be a simple object of $\cat$.
\begin{enumerate}
\renewcommand{\labelenumi}{{\rm (\alph{enumi})}}
\item The following assertions are equivalents:
\begin{enumerate}
\renewcommand{\labelenumii}{{\rm (\roman{enumii})}}
\item $V$ is left ambidextrous;
\item For all $f \in \End_\cat(V^* \otimes V)$, $\ev_V f=\tev_{V^*}f^*(\id_{V^*} \otimes \phi_V)$;
\item For all $f \in \End_\cat(V^* \otimes V)$, $f \tcoev_V =(\id_{V^*} \otimes \phi_V^{-1})f^* \coev_{V^*}$.
  \end{enumerate}
\item The following assertions are equivalents:
\begin{enumerate}
\renewcommand{\labelenumii}{{\rm (\roman{enumii})}}
\item $V$ is right ambidextrous;
\item For all $f \in \End_\cat(V \otimes V^*)$, $\tev_V f=\ev_{V^*}f^*(\phi_V \otimes \id_{V^*})$;
\item For all $f \in \End_\cat(V \otimes V^*)$, $f \coev_V=(\phi_V^{-1}\otimes
    \id_{V^*}) f^*\tcoev_{V^*}$.
  \end{enumerate}
 \item $V$ is spherical if and only if
   $$\tr_r^Y\!\bigl((\ev_V \otimes \id_Y) f \bigr)= \tr_r^{Y^*}\!\bigl((\id_{Y^*}
   \otimes \tev_{V^*})f^* \bigr)(\id_{V^*} \otimes \phi_V)$$ for all $Y \in
   \cat$ and $f \in \Hom_\cat(Y \otimes V^* \otimes V, V^* \otimes V \otimes
   Y)$.
\end{enumerate}
\end{proposition}
\begin{proof}
  For $Y \in \cat$, let $I_Y\colon\Hom_\cat(Y \otimes V^* \otimes V, V^*
  \otimes V \otimes Y) \to \End_\cat(V \otimes Y \otimes V^*)$ be the \kt
  linear isomorphism defined by
  $$
  I_Y(f)=\, \psfrag{V}[Bc][Bc]{\scalebox{.8}{$V$}}
  \psfrag{Y}[Bc][Bc]{\scalebox{.8}{$Y$}}\psfrag{f}[Bc][Bc]{\scalebox{.9}{$f$}}
  \rsdraw{.45}{.9}{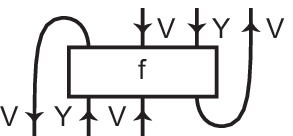}.
  $$

  Let us prove Part (a). The form $\brk{\,}_V$ is a left ambidextrous trace on
  $\{V\}$ if and only if \eqref{lambi} applied to $I_\unit(f)$ is satisfied
  for all $f \in \End_\cat(V^* \otimes V)$. Since $\alpha=\brk{\alpha}_V
  \id_V$ for all $\alpha \in \End_\cat(V)$, the condition \eqref{lambi}
  applied to $I_\unit(f)$ is equivalent to $\ev_V f=\tev_{V^*}f^*(\id_{V^*}
  \otimes \phi_V)$. Therefore, (i) is equivalent to (ii). Consequently (i) is
  equivalent to (iii), since (iii) is (ii) applied to the opposite category
  $\cat^\opp$ and an object is left ambidextrous in $\cat$ if and only if it
  is left ambidextrous in $\cat^\opp$.

  Part (b) is deduced from Part (a) by using $\cat^\rev$.  Let us prove Part (c). The form $\brk{\,}_V$ is a spherical trace on $\{V\}$
  if and only if \eqref{ambi1} applied to $I_Y(f)$ is satisfied for all $Y \in
  \cat$ and $f \in \Hom_\cat(Y \otimes V^* \otimes V, V^* \otimes V \otimes
  Y)$, which turns out to be equivalent to $\tr_r^Y\!\bigl((\ev_V \otimes
  \id_Y) f \bigr)=\tr_r^{Y^*}\!\bigl((\id_{Y^*} \otimes \tev_{V^*})f^*
  \bigr)(\id_{V^*} \otimes \phi_V) $ since $V$ is simple.
\end{proof}

\subsection{Proof of Theorem~\ref{thm:sided-traces}}\label{proof-thm-sided-trace}
We prove the theorem for left traces.  Then the statements for right
traces can be deduced using $\cat^\rev$. Let $t$ be a left ambidextrous trace on a class $\oo \subset \cat$. For $U \in \ideal^l_\oo$ and $\alpha\in\End_\cat(U)$, set
$$
\tzz_U(\alpha)=t_X\bigl(\tr_l^Y(q\alpha p)\bigr)
$$
where $X \in \oo$, $Y \in \cat$, $p\colon Y \otimes X \to U$, and $q \colon U \to Y \otimes X$ are such that $pq=\id_U$. We first verify that $\tzz_U(\alpha)$ does not depend on the choice of $p,q$. Let $p'\colon Y' \otimes X' \to U$ and $q' \colon U \to Y' \otimes X'$, with $X' \in \oo$, $Y' \in \cat$, such that $pq=\id_U$. Set
\begin{equation*}
f= \psfrag{Y}[Bc][Bc]{\scalebox{.8}{$Y$}}
   \psfrag{X}[Bc][Bc]{\scalebox{.8}{$X'$}}
   \psfrag{U}[Bc][Bc]{\scalebox{.8}{$U$}}
   \psfrag{A}[Bc][Bc]{\scalebox{.8}{$X$}}
   \psfrag{B}[Bc][Bc]{\scalebox{.8}{$Y'$}}
   \psfrag{l}[Bc][Bc]{\scalebox{.9}{$\alpha$}}
   \psfrag{r}[Bc][Bc]{\scalebox{.9}{$q'$}}
   \psfrag{s}[Bc][Bc]{\scalebox{.9}{$p$}}
   \psfrag{a}[Bc][Bc]{\scalebox{.9}{$q$}}
   \psfrag{b}[Bc][Bc]{\scalebox{.9}{$p'$}}
   \rsdraw{.45}{.9}{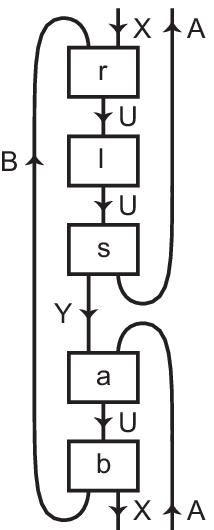}\colon X' \otimes X^* \to X' \otimes X^* ,
\end{equation*}
so that $\tr_r^{X^*}(f)= \tr_l^{Y'}(q'\alpha p')$ and $\phi_X^{-1} (\tr_l^{X'}(f))^* \phi_X=\tr^{Y}_l(q\alpha p)$. Therefore
$$
t_{X'}\bigl(\tr_l^{Y'}(q'\alpha p')\bigr)=t_{X'}\bigl(\tr_r^{X^*}(f)\bigr)=t_{X}\bigl(\phi_X^{-1} (\tr_l^{X'}(f))^* \phi_X\bigr)
=t_{X}\bigl(\tr^{Y}_l(q\alpha p)\bigr)
$$
and so $\tzz_U$ is well-defined. Clearly, $\tzz_{|\oo}=t$. Let us show that $\tzz$ is a left trace on $\ideal^l_\oo$. Let $Z \in \cat$, $U \in \ideal^l_\oo$, and $\alpha \in \End_\cat(Z \otimes U)$. Take $p\colon Y \otimes X \to U$ and $q \colon U \to Y \otimes X$, with $X \in \oo$ and $Y \in \cat$, such that $pq=\id_U$. Then $(\Id_Z \otimes p)(\Id_Z \otimes q)=\Id_{Z \otimes U}$. Therefore
$$
\tzz_{Z \otimes U}(\alpha)=t_X\bigl(\tr_l^{Z \otimes Y}((\Id_Z \otimes q)\alpha (\Id_Z \otimes p))\bigr)=
t_X\bigl(\tr_l^{Y}\bigl(q\tr_l^{Z}(\alpha)p\bigr)\bigr)=\tzz_U\bigl(\tr_l^{Z}(\alpha)\bigr).
$$
Let now $U,V \in \ideal_\oo^l$, $\alpha \in \Hom_\cat(U,V)$, and $\beta \in \Hom_\cat(V,U)$. Take $p\colon Y \otimes X \to U$, $q \colon U \to Y \otimes X$, $p'\colon Y' \otimes X' \to V$, $q' \colon V \to Y' \otimes X'$, with $X,X' \in \oo$ and $Y,Y' \in \cat$, such that $pq=\id_U$ and $p'q'=\id_V$. Set
\begin{equation*}
f= \psfrag{Y}[Bc][Bc]{\scalebox{.8}{$Y$}}
   \psfrag{X}[Bc][Bc]{\scalebox{.8}{$X'$}}
   \psfrag{U}[Bc][Bc]{\scalebox{.8}{$U$}}
   \psfrag{V}[Bc][Bc]{\scalebox{.8}{$V$}}
   \psfrag{A}[Bc][Bc]{\scalebox{.8}{$X$}}
   \psfrag{B}[Bc][Bc]{\scalebox{.8}{$Y'$}}
   \psfrag{g}[Bc][Bc]{\scalebox{.9}{$\alpha$}}
   \psfrag{h}[Bc][Bc]{\scalebox{.9}{$\beta$}}
   \psfrag{r}[Bc][Bc]{\scalebox{.9}{$q'$}}
   \psfrag{s}[Bc][Bc]{\scalebox{.9}{$p$}}
   \psfrag{a}[Bc][Bc]{\scalebox{.9}{$q$}}
   \psfrag{b}[Bc][Bc]{\scalebox{.9}{$p'$}}
   \rsdraw{.45}{.9}{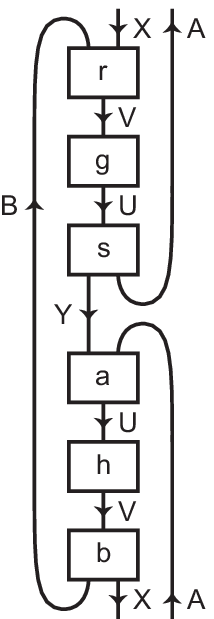}\colon X' \otimes X^* \to X' \otimes X^* ,
\end{equation*}
so that $\tr_r^{X^*}(f)= \tr_l^{Y'}(q'\alpha \beta p')$ and $\phi_X^{-1} (\tr_l^{X'}(f))^* \phi_X=\tr^{Y}_l(q\beta\alpha p)$.
Therefore
\begin{align*}
\tzz_V(\alpha\beta)&=t_{X'}\bigl(\tr_l^{Y'}(q'\alpha \beta p')\bigr)=t_{X'}\bigl(\tr_r^{X^*}(f)\bigr)\\
&=
t_X\bigl(\phi_X^{-1} (\tr_l^{X'}(f))^* \phi_X\bigr)=t_X\bigl(\tr^{Y}_l(q\beta\alpha p)\bigr)=\tzz_U\bigl(\beta\alpha).
\end{align*}
Hence, $\tzz$ is a left trace on $\ideal^l_\oo$. Suppose finally $\ell$ is another left trace on $\ideal^l_\oo$ with $\ell_{|\oo}=t$. Let $U \in \ideal^l_\oo$ and $\alpha\in\End_\cat(U)$. Take $p\colon Y \otimes X \to U$ and $q \colon U \to Y \otimes X$, with $X \in \oo$, $Y \in \cat$, such that $pq=\id_U$. Note that $U$ and $Y \otimes X$ belong to $\ideal^l_\oo$. Then
$$
\ell_U(\alpha)=\ell_U(\alpha pq)=\ell_{Y \otimes X}(q \alpha p)=\ell_X\bigl(\tr_l^Y(q\alpha p)\bigr)=t_X\bigl(\tr_l^Y(q\alpha p)\bigr)=\tzz_U(\alpha).
$$
Hence, $\ell=\tzz$.

\subsection{Proofs of Lemma~\ref{lem-ambi1} and Theorem~\ref{thm:traces}}\label{proof-thm-trace}
Let us prove Theorem~\ref{thm:traces} by taking Condition (i) of Lemma~\ref{lem-ambi1} as the
definition of a spherical trace.  Let $t$ be a spherical trace on a class
$\oo \subset \cat$. For $U \in \ideal_\oo$ and $\alpha\in\End_\cat(U)$, set
$$
\tzz_U(\alpha)=t_X\bigl(\tr_l^Y\!\tr_r^Z(q\alpha p)\bigr)
$$
where $X \in \oo$, $Y,Z \in \cat$, $p\colon Y \otimes X \otimes Z \to U$, and $q \colon U \to Y \otimes X \otimes Z$ are such that $pq=\id_U$. We first verify that $\tzz_U(\alpha)$ does not depend on the choice of $p,q$. Let $p'\colon Y' \otimes X' \otimes Z'\to U$ and $q' \colon U \to Y' \otimes X' \otimes Z'$, with $X' \in \oo$ and $Y',Z' \in \cat$, such that $pq=\id_U$. Set
\begin{equation*}
f= \psfrag{Y}[Bc][Bc]{\scalebox{.8}{$Y$}}
   \psfrag{X}[Bc][Bc]{\scalebox{.8}{$X$}}
   \psfrag{Z}[Bc][Bc]{\scalebox{.8}{$Z$}}
   \psfrag{U}[Bc][Bc]{\scalebox{.8}{$U$}}
   \psfrag{A}[Bc][Bc]{\scalebox{.8}{$X'$}}
   \psfrag{B}[Bc][Bc]{\scalebox{.8}{$Y'$}}
   \psfrag{C}[Bc][Bc]{\scalebox{.8}{$Z'$}}
   \psfrag{l}[Bc][Bc]{\scalebox{.9}{$\alpha$}}
   \psfrag{r}[Bc][Bc]{\scalebox{.9}{$q'$}}
   \psfrag{s}[Bc][Bc]{\scalebox{.9}{$p$}}
   \psfrag{a}[Bc][Bc]{\scalebox{.9}{$q$}}
   \psfrag{b}[Bc][Bc]{\scalebox{.9}{$p'$}}
   \rsdraw{.45}{.9}{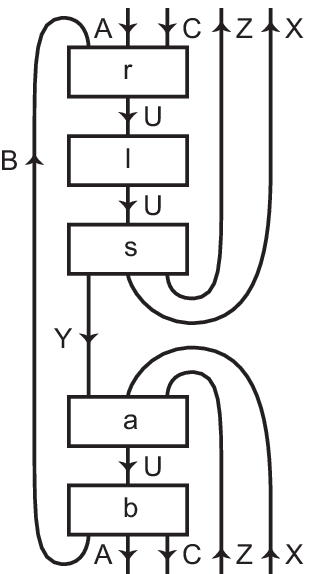}\colon X' \otimes Z' \otimes Z^* \otimes X^* \to X' \otimes Z' \otimes Z^* \otimes X^* ,
\end{equation*}
so that $$\tr_r^{Z' \otimes Z^* \otimes X^*}\!(f)= \tr_l^{Y'}\!\tr_r^{Z'}\!(q'\alpha p') \quad \text{and} \quad \phi_X^{-1} \bigl(\tr_l^{X' \otimes Z' \otimes Z^*}(f)\bigr)^* \phi_X=\tr^{Y}_l\!\tr_r^Z(q\alpha p).$$ Therefore
\begin{align*}
t_{X'}\bigl(\tr_l^{Y'}\!\tr_r^{Z'}\!&(q'\alpha p')\bigr)=t_{X'}\bigl(\tr_r^{Z' \otimes Z^* \otimes X^*}\!(f)\bigr)\\
&=t_{X}\bigl(\phi_X^{-1} \bigl(\tr_l^{X' \otimes Z' \otimes Z^*}(f)\bigr)^* \phi_X\bigr)
=t_{X}\bigl(\tr^{Y}_l\!\tr_r^Z(q\alpha p)\bigr)
\end{align*}
and so $\tzz_U$ is well-defined. Clearly, $\tzz_{|\oo}=t$. Let us show that $\tzz$ is a trace on $\ideal_\oo$. Let $A \in \cat$, $U \in \ideal_\oo$, $\alpha \in \End_\cat(A \otimes U)$ and $\beta \in \End_\cat(U \otimes A)$. Take $p\colon Y \otimes X \otimes Z\to U$ and $q \colon U \to Y \otimes X \otimes Z$, with $X \in \oo$ and $Y,Z \in \cat$, such that $pq=\id_U$. Then $(\Id_A \otimes p)(\Id_A \otimes q)=\Id_{A \otimes U}$ and $(p \otimes \Id_A)(q \otimes \Id_A)=\Id_{U\otimes A}$. Therefore
\begin{align*}
\tzz_{A \otimes U}(\alpha)&=t_X\bigl(\tr_l^{A \otimes Y}\tr_r^Z((\Id_A \otimes q)\alpha (\Id_A \otimes p))\bigr)\\ &=
t_X\bigl(\tr_l^{Y}\tr_r^Z\bigl(q\tr_l^{A}(\alpha)p\bigr)\bigr)=\tzz_U\bigl(\tr_l^{A}(\alpha)\bigr)
\end{align*}
and
\begin{align*}
\tzz_{U \otimes A}(\beta)&=t_X\bigl(\tr_l^{Y}\tr_r^{Z \otimes A}((q \otimes \Id_A)\beta (p \otimes \Id_A))\bigr)\\ &=
t_X\bigl(\tr_l^{Y}\tr_r^Z\bigl(q\tr_r^{A}(\beta)p\bigr)\bigr)=\tzz_U\bigl(\tr_r^{A}(\beta)\bigr).
\end{align*}
Let now $U,V \in \ideal_\oo$, $\alpha \in \Hom_\cat(U,V)$, and $\beta \in \Hom_\cat(V,U)$. Take $p\colon Y \otimes X \otimes Z\to U$, $q \colon U \to Y \otimes X\otimes Z$, $p'\colon Y' \otimes X' \otimes Z'\to V$, $q' \colon V \to Y' \otimes X'\otimes Z'$, with $X,X' \in \oo$ and $Y,Y',Z,Z' \in \cat$, such that $pq=\id_U$ and $p'q'=\id_V$. Set
\begin{equation*}
f= \psfrag{Y}[Bc][Bc]{\scalebox{.8}{$Y$}}
   \psfrag{X}[Bc][Bc]{\scalebox{.8}{$X$}}
   \psfrag{Z}[Bc][Bc]{\scalebox{.8}{$Z$}}
   \psfrag{U}[Bc][Bc]{\scalebox{.8}{$U$}}
   \psfrag{V}[Bc][Bc]{\scalebox{.8}{$V$}}
   \psfrag{A}[Bc][Bc]{\scalebox{.8}{$X'$}}
   \psfrag{B}[Bc][Bc]{\scalebox{.8}{$Y'$}}
   \psfrag{C}[Bc][Bc]{\scalebox{.8}{$Z'$}}
   \psfrag{l}[Bc][Bc]{\scalebox{.9}{$\alpha$}}
    \psfrag{t}[Bc][Bc]{\scalebox{.9}{$\beta$}}
   \psfrag{r}[Bc][Bc]{\scalebox{.9}{$q'$}}
   \psfrag{s}[Bc][Bc]{\scalebox{.9}{$p$}}
   \psfrag{a}[Bc][Bc]{\scalebox{.9}{$q$}}
   \psfrag{b}[Bc][Bc]{\scalebox{.9}{$p'$}}
   \rsdraw{.45}{.9}{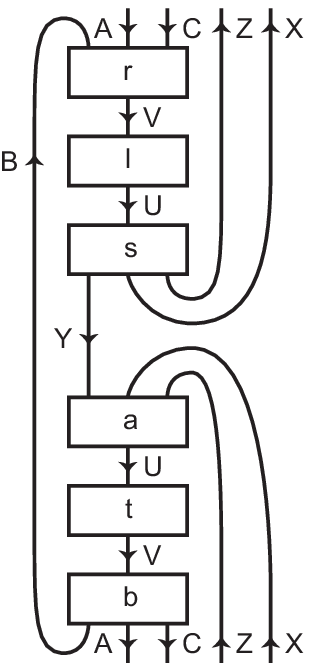}\colon X' \otimes Z' \otimes Z^* \otimes X^* \to X' \otimes Z' \otimes Z^* \otimes X^* ,
\end{equation*}
so that
\begin{align*}
&\tr_r^{Z' \otimes Z^* \otimes X^*}\!(f)= \tr_l^{Y'}\!\tr_r^{Z'}\!(q'\alpha \beta  p'),\\
 &\phi_X^{-1} \bigl(\tr_l^{X' \otimes Z' \otimes Z^*}(f)\bigr)^* \phi_X=\tr^{Y}_l\!\tr_r^Z(q\beta \alpha p).
\end{align*}
Therefore
\begin{align*}
\tzz_V(\alpha\beta)&=t_{X'}\bigl(\tr_l^{Y'}\!\tr_r^{Z'}\!(q'\alpha \beta  p')\bigr)=t_{X'}\bigl(\tr_r^{Z' \otimes Z^* \otimes X^*}\!(f)\bigr)\\
&=t_{X}\bigl(\phi_X^{-1} \bigl(\tr_l^{X' \otimes Z' \otimes Z^*}(f)\bigr)^* \phi_X\bigr)
=t_{X}\bigl(\tr^{Y}_l\!\tr_r^Z(q\beta \alpha p)\bigr)=\tzz_U(\beta\alpha).
\end{align*}
Hence, $\tzz$ is a trace on $\ideal_\oo$. Suppose finally $\ell$ is another trace on $\ideal_\oo$ with $\ell_{|\oo}=t$. Let $U \in \ideal_\oo$ and $\alpha\in\End_\cat(U)$. Take $p\colon Y \otimes X \otimes Z\to U$ and $q \colon U \to Y \otimes X \otimes Z$, with $X \in \oo$ and $Y,Z \in \cat$, such that $pq=\id_U$. Note that $U$ and $Y \otimes X$ belong to~$\ideal_\oo$. Then
\begin{align*}
\ell_U(\alpha)&=\ell_U(\alpha pq)=\ell_{Y \otimes X \otimes Z}(q \alpha p)=\ell_{Y \otimes X}\bigl(\tr_r^Z(q\alpha p)\bigr)\\
&=\ell_{X}\bigl(\tr_l^Y\tr_r^Z(q\alpha p)\bigr)=t_X\bigl(\tr_l^Y\tr_r^Z(q\alpha p)\bigr)=\tzz_U(\alpha).
\end{align*}
Hence, $\ell=\tzz$. 

Finally, let us prove Lemma~\ref{lem-ambi1}. Assume the family $t$ satisfies Condition (i) of Lemma~\ref{lem-ambi1}. By the above proof of Theorem~\ref{thm:traces}, there exists a trace $\tzz$ on $\ideal_\oo$ such that $\tzz_{|\oo}=t$ and $\tzz$ satisfies \eqref{ambi2} for all $X,X' \in \ideal_\oo$ and $Y \in \cat$. Since $\oo \subset \ideal_\oo$ and $\tzz_{|\oo}=t$, we have that $t$ satisfies \eqref{ambi2} for all $X,X' \in \oo$ and $Y \in \cat$. Hence, Condition (i) implies Condition (ii). Applying this implication to $\cat^\rev$ gives the reverse implication, since the ideal generated by $\oo$ in $\cat^\rev$ coincide with the ideal generated by $\oo$ in $\cat$. This concludes the proof of Lemma~\ref{lem-ambi1}.

\section{The case of projectives and the slope} \label{S:Slope}
The invariant of Theorem~\ref{thm-inv-graph2} relies on a certain set $\A$ of objects defined from a one-sided trace on a one-sided ideal.
In this section we give a characterization of $\A$, in terms of the slope, when the one-sided ideal is the ideal of projective objects.

Let $\cat$ be a category. Recall that an object $P$ of $\cat$ is \emph{projective} if the functor $\Hom_\cat(P,-)\colon \cat \to \mathrm{Set}$ preserves epimorphisms, that is, if for any epimorphism $p\colon X \to Y$ and any morphism $f\colon P \to Y$ in $\cat$, there exists a morphism $g \colon P \to X$ in~$\cat$ such that $f=pg$. We denote by $\Proj(\cat)$ the class of projective objects of $\cat$.

An object $Q$ of $\cat$ is \emph{injective} if it is projective in the opposite category $\cat^\mathrm{op}$. In other words, an object $Q$ of $\cat$ is injective if for any monomorphism $i\colon X \to Y$ and any morphism $f\colon X \to Q$ in $\cat$, there exists a morphism $g \colon Y \to Q$ in $\cat$ such that $f=gi$.
\begin{lemma}\label{L:projIdeal}
Let $\cat$ be pivotal category.
\begin{enumerate}
\renewcommand{\labelenumi}{{\rm (\alph{enumi})}}
\item $\Proj(\cat)$ is an ideal of $\cat$. In particular, $\Proj(\cat)^*=\Proj(\cat)$.
\item $\Proj(\cat)$ is the set of injective objects of $\cat$.
\item If $\ideal$ is a left (resp.\@ right) ideal of $\cat$ containing an object $V$ such that the left evaluation $\ev_V\colon V^* \otimes V \to \unit$ (resp.\@ the right evaluation $\tev_V\colon V \otimes V^* \to \unit$) is an epimorphism, then $\Proj(\cat) \subset \ideal$.
\item If $P$ is a projective object such that  $\ev_P$ (resp.\@ $\tev_P$) is an epimorphism, then $\ideal^l_{P}=\ideal_{P}=\Proj(\cat)$ (resp.\@ $\ideal^r_{P}=\ideal_{P}=\Proj(\cat)$).
\end{enumerate}
\end{lemma}
\begin{proof}
Let us prove Part (a). Let $P\in \Proj(\cat)$, $X \in \cat$, $U$ be a retract of $X \otimes P$, and
$u\colon X \otimes P \to U$, $v \colon U \to X \otimes P$ such that $uv=\id_U$. 
Let $p \colon M \to N$ be an epimorphism and $f \colon U \to N$ be a morphism. Set $$f'=(\id_{X^*} \otimes fu)(\tcoev_X \otimes \id_P) \colon P \to X^* \otimes N.$$
Since $\id_{X^*} \otimes p$ is an epimorphism, there exists a morphism $g'\colon P \to X^* \otimes M$ such that $f'=(\id_{X^*} \otimes p)g'$. Set
$
g=(\tev_X \otimes \id_M)(\id_X \otimes g')v \colon U \to M
$.
Then
\begin{align*}
pg & = (\tev_X \otimes \id_N)(\id_X \otimes (\id_{X^*} \otimes p) g')v \\
&=(\tev_X \otimes \id_N)(\id_X \otimes f')v \\
&= (\tev_X \otimes f)
(\id_{X \otimes X^*} \otimes u)
(\id_X \otimes \tcoev_X \otimes \id_P)v\\
&=fuv=f.
\end{align*}
Therefore, $U \in \Proj(\cat)$, and so $\Proj(\cat)$ is a left ideal. Likewise one shows that $\Proj(\cat)$ is a right ideal. Hence, $\Proj(\cat)$ is an ideal. 
Thus, Lemma \ref{L:id*} implies that $\Proj(\cat)^*=\Proj(\cat)$.

Part (b) follows from the fact that the duality functor $\cat^\opp \to \cat$
is an equivalence. Let us prove the left version of Part (c), from which the
right version can be deduced by using $\cat^\rev$. Let $P \in
\Proj(\cat)$. since $\id_P \otimes \ev_V$ is an epimorphism, there exists a
morphism $g \colon P \to P \otimes V^* \otimes V$ such that
$(\id_P\otimes\ev_V)g =\id_P$. Therefore, $P$ is a retract of $P \otimes V^*
\otimes V$. Since $P \otimes V^* \otimes V \in \ideal$ as $\ideal$ is a left
ideal, we get $P \in \ideal$. Hence, $\Proj(\cat) \subset \ideal$.

Finally, let us prove the left version of Part (d), from which the right version can be deduced by using $\cat^\rev$.
By Part (c), we have: $\Proj(\cat)\subset \ideal^l_P$. Now $\ideal^l_P\subset \ideal_{P}$ since $\ideal_{P}$ is a left ideal containing $P$ and $\ideal_{P}\subset \Proj(\cat)$ since $\Proj(\cat)$ is an ideal containing $P$. Therefore,
$\ideal^l_P=\ideal_{P}=\Proj(\cat)$.
\end{proof}

Now let $\cat$ be a pivotal \kt category and $P$ be a projective object such
that $\ev_P$ is an epimorphism.
Assume there exists a non zero left ambidextrous trace $t$ on $\{P\}$. For
example, such a trace exists when $P$ is left ambidextrous (see
Section~\ref{sect-ambidextrous-object}). By Lemma~\ref{L:projIdeal},
$\Proj(\cat)=\ideal^l_P$ and so, by Theorem~\ref{thm:sided-traces}, there
exists a left trace $\tzz$ on $\Proj(\cat)$ such that
$\tzz_P=t$.  Denote by $\qd$ the modified
left dimension associated with $\tzz$, see Section \ref{SS-sided-traces}.
\begin{lemma}\label{L:idealProj}
Let $\tzz$ and $\qd$ be the left trace and modified left dimension associated to $P$ and $t$ as defined in the previous paragraph.
Let $V$ be a projective object such that $\ev_V$ an epimorphism. Then $\tzz_V\neq0$.  Moreover,  if $V$ is simple, then  $\qd(V)\neq0$.
\end{lemma}
\begin{proof}
Let $g \in \End_\cat(P)$ such that
$\tzz_P(g)=t(g) \neq 0$.  By Lemma~\ref{L:projIdeal}, $\ideal^l_V=\Proj(\cat)$. By Lemma~\ref{P:ideal4}, since $P \in \ideal^l_V$, there exists $f \in \End_\cat(P \otimes V^*)$ such that $\id_P=\tr_r^{V^*}\!(f)$. Set $h=\phi_V^{-1} \bigl(\tr_l^{P}\!((g \otimes \id_{V^*})f)\bigr)^* \phi_V \in \End_\cat(V)$.
By Theorem~\ref{thm:sided-traces}, $\tzz$ is a left ambidextrous trace on $\Proj(\cat)$. Therefore, Equation \eqref{lambi} implies
$$
\tzz_V(h)
=\tzz_{P}\bigl(\tr_r^{V^*}\!((g \otimes \id_{V^*})f)\bigr)
=\tzz_{P}\bigl(g \tr_r^{V^*}\!(f)\bigr)=\tzz_{P}(g) \neq 0.
$$
Hence, $\tzz_V\neq 0$.  Now if  $V$ is simple. Let $k \in \kk$ such that $h=k \id_V$. Then $k\,\qd(V)=k\,\tzz_V(\id_V)=\tzz_V(h)\neq 0$, and so $\qd(V) \neq 0$.
\end{proof}
Set
$$
\A=\{V \in \Proj(\cat) \,|\, \text{$V$ simple and $\ev_V$, $\tev_V$ are epimorphisms}\}.
$$
Note that $\A^*=\A$ since the dual of a simple object is simple, $\Proj(\cat)^*=\Proj(\cat)$, and for any $X\in\cat$, $\ev_{X^*}$ is an epimorphism if and only if $\tev_X$ is an epimorphism.

Assume now that $\kk$ is a field. For $V \in \A$, the \emph{slope} of $V$ is
$$
s(V)=\qd(V)/\qd(V^*) \in \kk^\times.
$$
The slope is well-defined since $\A^*=\A$ and $\qd(U) \neq 0$ for all $U \in \A$ (by Lemma~\ref{L:idealProj}). Recall $\tzz{\vee}$ is the right trace on $\Proj(\cat)^*=\Proj(\cat)$ defined by $\tzz_X^\vee(f)=\tzz_{X^*}(f^*)$ for $X \in \Proj(\cat)$ and $f \in \End_\cat(X)$.

\begin{proposition}\label{P:Slope}
  The slope $s\colon \A\to \kk^{\times}$ has the following properties:
  \begin{enumerate}
    \renewcommand{\labelenumi}{{\rm (\alph{enumi})}}
  \item $s(V^*)=s(V)^{-1}$ for all $V\in \A$.
  \item $s(U)=s(V)s(W)$ for all $U,V,W \in \A$ such that $U$ is a retract of $V\otimes W$.
  \item For any $V \in \A$, $\tzz{\vee}_V=\tzz_V$ if and only if $s(V)=1$. \label{I:sphiff}
  \end{enumerate}
\end{proposition}
\begin{proof}
  Part (a) is an immediate consequence of the definition since $\qd(V^{**})=\qd(V)$ because $V^{**} \simeq V$. Let us prove Part (c). Let  $f\in\End_\cat(V)$. There exists $\lambda \in \kk$ such that $f=\lambda\id_V$ and so $f^*=\lambda\id_{V^*}$. Then
\begin{equation}\label{eq-dem}
\tzz_V(f)=\lambda\qd(V)=\lambda\qd(V^{*})s(V)=s(V)\tzz_{V^{*}}(f^{*})=s(V)\tzz_V^\vee(f).
\end{equation}
Hence, $\tzz_V =s(V)\, \tzz_V^\vee$.

Let us prove Part (b).  Let $p\in \Hom_\cat(V\otimes W,U)$ and $q \in   \Hom_\cat(U,V\otimes W)$ such that $pq=\id_U$. Set
$$
a=(p \otimes \id_{W^*})(\id_V \otimes \coev_W)  \quad \text{and} \quad b=(\id_V \otimes \tev_W)(q \otimes \id_{W^*}) .
$$
Note that $\tr_r^{W^*}\!(ab)=pq=\id_U$ and $(ba)^*=\tr_l^{W^*}\!(p^*q^*)$. Since $\tzz$ is a left trace, and in particular satisfies \eqref{lambi} (see Theorem~\ref{thm:sided-traces}), and by using \eqref{eq-dem}, we have:
\begin{align*}
\qd(U)&=\tzz_U(\tr_r^{W^*}\!(ab))
=\tzz_W( \phi_W^{-1} \bigl(\tr_l^{U}\!(ab)\bigr)^*\phi_W) \\
&=s(W)\, \tzz_{W^*}( (\phi_W^{-1} \bigl(\tr_l^{U}\!(ab)\bigr)^*\phi_W)^*)
=s(W)\, \tzz_{W^*}(\tr_l^{U}\!(ab))\\
&=s(W)\, \tzz_{U \otimes W^*}(ab)
=s(W)\, \tzz_{V}(ba)\\
&=s(V)s(W)\, \tzz_{V^*}((ba)^*)
=s(V)s(W)\, \tzz_{V^*}(\tr_l^{W^*}\!(p^*q^*))\\
&=s(V)s(W)\, \tzz_{W^* \otimes V^*}(p^*q^*)
=s(V)s(W)\, \tzz_{U^*}(q^*p^*)\\
&=s(V)s(W)\, \tzz_{U^*}(\id_{U^*})
=s(V)s(W)\, \qd(U^*).
\end{align*}
Hence, $s(U)=s(V)s(W)$.
\end{proof}
Remark that when $\cat$ is ribbon, meaning that $\cat$ is endowed with a braiding $\tau$ such that its associated twist
$$
\theta=\{\theta_X=\tr_r^{X}(\tau_{X,X})\colon X \to X\}_{X \in \cat}
$$
is self-dual (i.e., $(\theta_X)^*=\theta_{X^*}$ for all $X \in \cat$), then $s(V)=1$ for all $V \in \A$ (see \cite{GPT1}), and so $\tzz\vee=\tzz$ on $\A$.

\vfill
\end{document}